\documentclass[10pt]{amsart}
\usepackage[english]{babel}
\usepackage{amssymb}
\usepackage{amsmath}
\newcommand{\D}{\mathcal{D}}

\newcommand{\J}{\mathcal{J}}
\newcommand{\E}{\mathcal{E}}
\newcommand{\F}{\mathcal{F}}
\newcommand{\G}{\mathcal{G}}
\newcommand{\h}{\mathcal{H}}
\newcommand{\s}{\mathcal{S}}
\newcommand{\A}{\mathcal{A}}
\newcommand{\M}{\mathcal{M}}
\newcommand{\N}{\mathcal{N}}
\newcommand{\Pe}{\mathcal{P}}
\newcommand{\B}{\mathcal{B}}
\newcommand{\comp}{\mathcal{C}}
\newcommand{\sts}{\mathcal{O}_X}
\newcommand{\Pic}{{\rm Pic}}

\newcommand{\quat}{\mathcal{D}}
\newcommand{\lb}{\mathcal{L}}
\newcommand{\torus}{\mathcal{T}}

\newtheorem{lemma}{Lemma}
\newtheorem{theorem}{Theorem}
\newtheorem{proposition}{Proposition}
\newtheorem{corollary}{Corollary}

\theoremstyle{definition}

\newtheorem{example}{Example}

\newtheorem{remark}{Remark}



\title{Albert algebras over curves of genus zero and one}

\date{10.9.2007}
\author{S. Pumpl\"un}
\email{susanne.pumpluen@nottingham.ac.uk}
\address{School of Mathematics\\
University of Nottingham\\
University Park\\
Nottingham NG7 2RD\\
United Kingdom}

\begin{document}

\begin{abstract}
 Albert algebras and other Jordan algebras are constructed over curves of genus zero and one,
  using a generalization of the Tits process and the first Tits construction due to Achhammer.
\end{abstract}

\maketitle

\section*{Introduction}

In the structure theory of Jordan algebras, exceptional simple Jordan algebras ({\it Albert algebras}) are
as important as octonion algebras are in the structure theory of alternative algebras.
In his PhD thesis [Ach1, 2], Achhammer developed a generalized Tits process for algebras over a locally ringed space
which, applied to algebras over arbitrary rings,  generalized  the classical Tits process for algebras over base rings
 by Petersson-Racine [P-S1, 2] and thus, in particular,
the classical first Tits construction. He obtained general results on Albert algebras over locally ringed spaces
and found
examples of Albert algebras over $n$-dimensional projective space.  Achhammer's PhD
thesis was only partially published, hence some of his results have been written up in the first sections of [Pu2].

Independently, Albert
algebras over integral schemes were investigated by
Parimala-Suresh-Thakkur [P-S-T1, 2]. In [P-S-T2], generalized first and second Tits constructions to obtain Albert algebras
were introduced, starting with an Azumaya algebra of rank 9 over a domain $R$ such that $2,3\in R^\times$.
 Over integral schemes, the first Tits construction of [P-S-T2] coincides with the one given in [Ach1, 2] when starting
with an Azumaya algebra of constant rank 9, the second one with the Tits process in [Ach1, 2] when starting with
the subsheaf of hermitian elements of an Azumaya algebra of constant rank 9 over a suitable sheaf of rings
 $\mathcal{O}_X'$, which carries an involution
$*$ such that $\mathcal{H}(\mathcal{O}_X',*)=\mathcal{O}_X$.
Albert algebras over $\mathbb{A}_k^2$  which do not arise from a generalized first or second
 Tits construction were constructed in [P-S-T2].

We construct Jordan algebras  over curves of genus zero and genus one
using Achhammer's Tits process and the general first Tits construction as explained in [Ach1, 2].
It is interesting to see which (selfdual)
 vector bundles can carry the structure of an Albert algebra, and if there are any such bundles consisting of
 indecomposable summands of large ranks.
Since the automorphism groups of these algebras are algebraic group schemes of type $F_4$ over $X$, the results obtained
also contain information on the structure of these groups schemes.

 The vector bundles over curves of genus zero are of relatively simple
type (there is no indecomposable vector bundle of rank greater than 2, and the only absolutely indecomposable ones are the line
bundles). Thus it is worth going one step further and looking at algebras curves of genus one. The vector bundles over elliptic curves
and their behaviour are well-known [At, AEJ1, 2, 3]. There exist absolutely indecomposable selfdual vector bundles
 of arbitrary rank, and indecomposable ones of different types.
Over elliptic curves (or more generally, over curves of genus one, since, indeed,
our arguments generally would work also for curves without rational points), a classification
of Albert algebras seems to be still out of reach. One of the problems is that the Azumaya algebras of constant
rank 3 over an elliptic curve are not sufficiently know yet, at least to the author's knowledge, which would
be important in order to get an idea on all the first Tits constructions which are possible
starting with such an algebra. There are some results on  octonion algebras [Pu1], which help to give examples
of Albert algebras of the type $\h_3(\comp)$, but again lacking a full classification of these algebras, we
still miss some crucial information. Another stumbling block is to actually compute the ingredients for a first Tits
construction or a Tits process. Here the vast choice of vector bundles possible over an elliptic curve creates some
difficulties.
 We  construct a large number of non-isomorphic Albert algebras
by applying the first Tits construction to Azumaya algebras of the type $\E nd_X(\E)$.
The fact that the Theorem of Krull-Schmidt holds both over curves of genus zero and one makes it usually
 easy to decide which
Albert algebras must be non-isomorphic.

In Section 1 we establish some notation and collect the basic facts needed later. Algebras over curves of genus zero are
studied in Sections 2, 3 and 4; in particular, Albert algebras over curves of genus zero which are obtained by a first
Tits construction in Section 3.
The Albert algebras over the projective line $\mathbb{P}_k^1$, with $k$ a base field of characteristic zero,
which can be realized as a first Tits construction starting with an Azumaya algebra over $\mathbb{P}_k^1$ of constant rank
9, have been described in [Ach1, 4.7] (see Theorem 1).
Over nonrational curves $X$ of genus zero, we describe all first Tits constructions starting with an
Azumaya algebra over $X$ of rank 9 which is defined over $k$ (Corollary 2).
All Albert algebras which can be obtained by a first Tits construction
starting with an Azumaya algebra $\A$ over $X$ such
that, at the generic point $\xi$, $\A(\xi)$ is a division algebra, are defined over $k$ (Corollary 3).
We then give examples of Albert algebras constructed by the Tits process in Section 4.
 In Section 5, the basic terminology for elliptic curves is
 introduced. \'Etale First Tits constructions over elliptic curves are studied in Section 6 and
  Albert algebras over elliptic curves of the kind $\h_3(\comp)$ in Section 7.
 Albert algebras over elliptic curves obtained by a first Tits construction starting with an Azumaya algebra
 of the type $\E nd_X(\E)$ are constructed in Section 8.

We use the standard terminology from algebraic geometry, see Hartshorne's book [H] and the one for algebras
 developed in [P]. For the standard terminology on Jordan algebras, the reader is referred to the books by
McCrimmon [M2], Jacobson [J] and Schafer [Sch].

\section{Preliminaries}

In the following, let $(X,\sts)$ be a locally ringed space and $k$ a field.

\subsection{Algebras over $X$}

 For $P \in X$ let $\mathcal{O}_{P,X}$ be the local ring of
$\mathcal{O}_X$ at $P$ and $m_P$ the maximal ideal of $\mathcal{O}_{P,X}$. The corresponding residue class
field is denoted by $k(P)=\mathcal{O}_{P,X}/m_P$. For an $\mathcal{O}_X$-module $\mathcal{F}$ the stalk of
$\mathcal{F}$ at $P$ is denoted by $\mathcal{F}_P$. $\mathcal{F}$ is said to have {\it full support} if
${\rm Supp}\,\mathcal{F}=X$; i.e., if $\mathcal{F}_P\not=0$ for all $P\in X$. We call
$\mathcal{F}$ {\it locally free of finite rank} if for each $P\in X$ there is an open neighborhood $U\subset X$
of $P$ such that $\mathcal{F}|_U=\mathcal{O}_U^r$ for some integer $r\geq 0$. The {\it rank} of
$\mathcal{F}$ is defined to be ${\rm sup}\{{\rm rank}_{\mathcal{O}_{P,X}}\mathcal{F}_P\,|\, P\in X\}$.
 The term ``$\mathcal{O}_X$-algebra" (or ``algebra over $X$'') always refers to nonassociative
$\mathcal{O}_X$-algebras which are unital and locally free of finite rank as $\mathcal{O}_X$-modules.
 An $\mathcal{O}_X$-algebra $\mathcal{A}$ is called {\it alternative} if $x^2y=x(xy)$ and
$yx^2=(yx)x$ for all sections $x,y$ of $ \mathcal{A}$ over the same open subset of $X$.

\subsection{Cubic forms}
Let $d$ be a positive integer.  Let $\mathcal{M}$, $\N$ be $\mathcal{O}_X$-modules which are locally free of finite rank.
 In general, that is if there are no restrictions on the global sections $H^0(X,\sts)$,
 polynomial maps from $\M$ to $\N$ are defined
analogously as it was done in [R], [L1, \S 18] or [L2, Section 1] for polynomial maps between modules over rings, see  [Ach1].
In that case forms of degree $d$ are usually identified with the induced polynomial maps.

If $2,3\in H^0(X,\mathcal{O}_X^\times)$, which will be the case considered in large parts of the paper,
 the definition for cubic forms is equivalent to the following:
 A {\it cubic form} on $\mathcal{M}$ over $\mathcal{O}_X$ is a map $N:\mathcal{M}\to \mathcal{O}_X$
 such that $N(a x)=a^3 N(x)$ for all sections $a$ in $\mathcal{O}_X$, $x$ in $\mathcal{M}$ over the same open subset
 of $X$, where the map $\theta : \mathcal{M} \times  \mathcal{M} \times \mathcal{M}\to \mathcal{O}_X$ defined by
 $$\theta(x,y,z)=\frac{1}{6}(\varphi(x+y+z)-\varphi(x+y)-\varphi(x+z)-\varphi(y+z)+ \varphi(x)+\varphi(y)+\varphi(z))$$
  is a trilinear form over $\mathcal{O}_X$.

For a cubic form $N:\mathcal{M}\to \mathcal{O}_X$ on a locally free
$\mathcal{O}_X$-module $\mathcal{M}$ of finite rank  with full support, $N$ is called {\it  nondegenerate}
 if $N(P):\mathcal{M}(P)\to k(P)$  is nondegenerate for all $P\in X$.
This notion of nondegeneracy is invariant under base change.

A trilinear form $\theta : \mathcal{M} \times  \mathcal{M} \times \mathcal{M}\to \mathcal{O}_X$ over $\mathcal{O}_X$
 is called {\it symmetric} if $\theta (x,y,z)$ is invariant under all permutations of its variables.
A locally free $\sts$-module $\mathcal{M}$ together with a symmetric trilinear map $\theta: \mathcal{M}\times
\mathcal{M}\times \mathcal{M}\to\sts$ is called a {\it trilinear space}.

 Two trilinear spaces $(\mathcal{M}_i,\theta_i)$ ($i=1,2$) are called {\it isomorphic}
if there exists an $\mathcal{O}_X $-module isomorphism $f:\mathcal{M}_1\to \mathcal{M}_2$ such that
$\theta_2(f(v_1),f(v_2),f(v_3))=\theta_1(v_1,v_2,v_3)$ for all sections $v_1,v_2,v_3$ of
$\mathcal{M}_1$ over the same open subset of $X$.
The {\it orthogonal sum} $(\mathcal{M}_1,\theta_1)\perp (\mathcal{M}_1,\theta_2)$ of two trilinear spaces
$(\mathcal{M}_i,\theta_i)$, $i=1,2$, is defined as the $\mathcal{O}_X$-module
 $\mathcal{M}_1\oplus \mathcal{M}_2$ together with the trilinear form
$(\theta_1 \perp \theta_2)(u_1+v_1,u_2+v_2,u_3+v_3)=\theta_1(u_1,u_2,u_3)+\theta_2(v_1,v_2,v_3)$.
If $2,3\in H^0(X,\mathcal{O}_X^\times)$ we canonically identify symmetric trilinear forms and cubic forms.

\subsection{Composition algebras over $X$}
 Following  [P], an $\mathcal{O}_X$-algebra $\mathcal{C}$ is called a {\it composition algebra} over $X$
if it has full support, and if there exists a quadratic form $N \colon \mathcal{C} \to  \mathcal{O}_X$ such that
the induced symmetric bilinear form $N(u,v)  =
N(u+v)-N(u)-N(v)$ is {\it nondegenerate}; i.e., it determines a module isomorphism
$\mathcal{C} \overset{\sim}{\longrightarrow} \check{\mathcal{C}}
=\mathcal{H}om (\mathcal{C},\mathcal{O}_X)$,
and such that $N(uv)=N(u)N(v)$ for all sections $u,v$ of $\mathcal{C}$ over the
same open subset of $X$.
 $N$ is uniquely determined by these conditions and called the {\it norm} of $\mathcal{C}$.
 It is denoted by $N_\mathcal{C}$.
 Composition algebras over $X$ are invariant under base change, and exist only in ranks 1, 2, 4 or 8.
A composition algebra of constant rank 2 (resp. 4 or 8) is called a {\it quadratic \'etale algebra} (resp. {\it quaternion algebra} or
{\it octonion algebra}). A composition algebra over $X$ of constant rank is called {\it split}
 if it contains a composition subalgebra
isomorphic to $\mathcal{O}_X \oplus \mathcal{O}_X$.  There are several construction methods for composition algebras
over locally ringed spaces:  There  exists a Cayley-Dickson process ${\rm Cay}(\D,\Pe,N_\Pe)$
which is described in [P1].

\subsection{Jordan algebras and cubic forms with adjoint and base point} ([Ach1], [P-R])
Let $\J$ be a locally free $\sts$-module of
finite rank. Following  [Ach1], $( \mathcal{J},U,1)$ with $1\in H^0(X,\mathcal{J})$ is a {\it Jordan algebra} over $X$ if
\begin{enumerate}
\item $U:\mathcal{J}\to\mathcal{E}nd_{\mathcal{O}_X}(\mathcal{J}), x\to U_x$ is a quadratic map;
\item $U_1=id_\mathcal{J}$;
\item $U_{U_x(y)}=U_x\circ U_y\circ U_x$ for all sections $x,y$ in $ \mathcal{J}$;
\item $U_{x}\circ U_{y,z}(x)=U_{x,U_x(z)}(y)$ for all sections $x,y,z$ in $ \mathcal{J}$;
\item For every commutative associative $\sts$-algebra $\sts'$, $\mathcal{J}\otimes \sts'$ satisfies $(3)$ and $(4)$.
\end{enumerate}
In general, we  write  $\J$ instead of $( \mathcal{J},U,1)$.

 An  $\mathcal{O}_X$-algebra $\mathcal{J}$ is called an {\it Albert algebra}
 if $\mathcal{J}(P)=J_P\otimes k(P)$ is an Albert algebra over $k(P)$
for all $P\in X$. If $\mathcal{J}$ is a Jordan algebra over a scheme $(X,\mathcal{O}_X)$
then $\mathcal{J}$ is an Albert algebra if and only if there is a covering
$V_i\to X$ in the flat topology on $X$ such that $\mathcal{J}\otimes \mathcal{O}_{V_i}\cong \mathcal{H}_3({\rm Zor}( \mathcal{O}_{V_i}))$
where ${\rm Zor}(  \mathcal{O}_{V_i})$ denotes the octonion algebra of Zorn vector matrices over  $\mathcal{O}_{V_i}$
and  $\mathcal{H}_3({\rm Zor}( \mathcal{O}_{V_i}))$ is the Jordan algebra of 3-by-3 hermitian matrices
with entries in ${\rm Zor}( \mathcal{O}_{V_i})$ and scalars $ \mathcal{O}_{V_i}$  on the diagonal [Ach1, 1.10].

A tripel $(N,\sharp,1)$ is a {\it cubic form with adjoint and base point} on $\J$ if $N:\J\to \sts$ is a cubic form,
$\sharp:\J\to\J$ a quadratic map and $1\in H^0(X,\J)$, such that
$$\begin{array}{l}
 x^{\sharp\,\sharp}=N(x)x,\\
 T(x^\sharp,y)=D_yN(y) \text{ for } T(x,y)=-D_xD_y {\rm log} N(1),\\
 N(1)=1,\, 1^\sharp=1,\\
 1\times y=T(y)1-y \text{ with } T(y)=T(y,1),\, x\times y=(x+y)^\sharp-x^\sharp-y^\sharp
\end{array}$$
for all sections $x,y$ in $\J$ over the same open subset of $X$.
Here, $D_yN(x)$ denotes the
directional derivative of $N$ in the direction $y$, evaluated at $x$.
The symmetric bilinear form $T$ is called the {\it trace form} of $\J$.

Every cubic form with adjoint and base point $(N,\sharp,1)$ on a locally free $\sts$-module $\J$ of finite rank
defines a unital Jordan algebra structure $\mathcal{J}(N,\sharp,1)=(\J,U,1)$ on $\J$ via
$$U_x(y)=T(x,y)x-x^\sharp\times y$$
for all sections $x,y$ in $\J$, where the identities given in
[P-R, p.~213] hold for all sections in $\J$. A section $x\in\J$ is invertible iff $N(x)$
is invertible. In that case, $x^{-1}=N(x)^{-1}x^\sharp$.

\subsection{The  Tits process} (due to Achhammer [Ach1], see [Pu2] for details.)
Let $(X,\sts')$ be a locally ringed space. Let $*:\sts'\to \sts'$ be an involution, $\B$ an associative $\sts'$-algebra
and $*_\B$ an involution on $\B$ such that $*_\B|_{\sts'}=*$. We
write $*_\B$ instead of $*$ from now on, also for the involution $*$ on $\sts'$. Let $(N_\B,\sharp_\B, 1)$ be a
cubic form with adjoint and base point on $\B$.

Assume that $(N_\B,\sharp_\B, 1)$ is $\B$-{\it admissible}.
That means,
\begin{enumerate}
\item $\B^+=J(N_\B,\sharp_\B,1)$ with $1\in H^0(X,\B)$ the unit element in $\B$, and $xyx=T_\B(x,y)x-x^{\sharp_\B}\times_\B
y$ for $x,y$ in $\B$;
\item $N_\B(xy)=N_\B(x)N_\B(y)$ for all $x,y$ in $\B$;
\item $N_\B(x^{*_\B})=N_\B(x)^{*_\B}$ for $x$ in $\B$.
\end{enumerate}
 Write $\B^\times$ for the sheaf of units
of $\B$. The morphism of group
sheaves $N_\B:\B^\times\to \sts'^\times$ induces a morphism
$$\sharp_\B: {\rm Pic}_{l}\B\to  {\rm Pic}_{l}\B^{\rm op}.$$
A cubic map $f:\Pe \to \E$ in the category of $\sts'$-modules is called {\it multiplicative} if
$$f(bw)=N_\B(b)f(w)$$
 for all sections $b$ in $\B$, $w$ in $\E$. Let $\F$ be a right $\B$-module.
 A quadratic map $g:\Pe \to \F$ in the category of $\sts'$-modules is called {\it multiplicative} if
$$g(bv)=g(v)b^{\sharp_\B}$$
 for all sections $b$ in $\B$, $v$ in $\Pe$. A multiplicative cubic map $N:\Pe\to N_\B(\Pe)$ is called a {\it norm}
 on $\Pe$ if $N$ is universal in the category of multiplicative cubic maps on $\Pe$.
 A multiplicative quadratic map $\sharp:\Pe\to \Pe^{\sharp_\B}$ is called an {\it adjoint} on $\Pe$ if $\sharp$
 is universal in the category of multiplicative quadratic maps on $\Pe$.

Let $\Pe^\vee$ denote the right $\B$-module $\mathcal{H}om_\B(\Pe,\B)$.
Let $\Pe\in {\rm Pic}_{l}\B$ such that $N_\B(\Pe)\cong \sts'$ and let $N:\Pe\to \sts'$ be a cubic norm on $\Pe$.
Then
$$\Pe^{\sharp_B}\cong \Pe^\vee,\quad {\Pe^\vee}^{\sharp_B}\cong \Pe$$ and
$N_\B(\Pe^\vee)\cong\sts'$.
There exists a uniquely determined cubic norm $\check{N}:\Pe^\vee\to \sts'$ and
uniquely determined adjoints $\sharp:\Pe\to \Pe^\vee$ and $\check{\sharp}:\Pe^\vee\to \Pe$ such that
\begin{enumerate}
\item $\langle w,w^\sharp\rangle=N(w)1$;
\item $\langle \check{w}^{\check{\sharp}},\check{w}\rangle=\check{N}(\check{w})1$;
\item $w^{\sharp \, \check{\sharp}}=N(w)w$
\end{enumerate}
for all $w$ in $\Pe$, $\check{w}$ in $\Pe^\vee$. Moreover,
\begin{enumerate}
\item $\check{w}^{\check{\sharp}\,\sharp }=\check{N}(\check{w})\check{w}$;
\item $\langle w, \check{w}\rangle^{\sharp_\B}= \langle \check{w}^{\check{\sharp}}, w^{\sharp}\rangle$;
\item $N_\B(\langle w,w^\sharp\rangle)=N(w)\check{N}(\check{w})$;
\item $D_{w'}N(w)=T\B(\langle w', w^\sharp\rangle)$;
\item $D_{\check{w}'}\check{N}(\check{w})=T\B(\langle\check{w}^{\check{\sharp}},\check{ w}'\rangle)$;
\item $\langle w,\check{w}\rangle w= T_\B(\langle w,\check{w}\rangle)w-w^\sharp\check{\times}\check{w}$
\end{enumerate}
for all $w,w'$ in $\Pe$, $\check{w}$ in $\Pe^\vee$. The morphism of group sheaves
$$*_\B:\B^\times\rightarrow (\B^{\rm op})^\times$$
determines  morphisms
$$\begin{array}{l}
*_\B:{\rm Pic}_{l}\,\B=\check{H}^1(X,\B^\times)\rightarrow \check{H}^1(X,(\B^{op})^\times)={\rm Pic}_{l}\,\B^{\rm op},\\
*_\B:{\rm Pic}_{l}\,\B^{\rm op}=\check{H}^1(X,(\B^{\rm op})^\times)\rightarrow \check{H}^1(X,\B^\times)={\rm Pic}_{l}\,\B
\end{array}$$
of pointed sets.

For a locally free left $\B$-module $\Pe$ (respectively, right $\B$-module) of rank one, let $\overline{\Pe}$ denote the opposite
module of $\Pe$  with respect to the involution $*_\B$. An isomorphism of left
(respectively, right)  $\B$-modules $j:\Pe \to \overline{{\Pe^*}^{\B}}$ is called an {\it involution } on $\Pe$.
Furthermore, we canonically identify the left $\B$-homomorphisms from $\Pe$ to $\overline{\Pe^\vee}$ with the
sesquilinear forms on $\Pe$.

 A pair $(\A,\sts)$ consisting of a subsheaf of rings $\sts$
of $\sts'$ and an $\sts$-submodule $\A$ of $\B$ is called $\B$-{\it ample} if
\begin{enumerate}
\item $\sts\subset \mathcal{H}(\sts',*_{\B})$,
\item $rr^{*_\B}\in\sts$ for $r$ in $\sts'$,
\item $\A\subset \mathcal{H}(\B,*_{\B})$,
\item $1\in H^0(X,\A)$,
\item $bab^*\in \A$ for $a\in\A$, $b\in\B$,
\item $N_\B(\A)\subset\sts$; i.e., $N_\B|_\A:\A\to\sts$ is a cubic form over $\sts$,
\item $\A^{\sharp_\B}\subset \A$; i.e., ${\sharp_\B}|_\A: \A\to\A$ is a quadratic map over $\sts$.
\end{enumerate}
Let $(\A,\sts)$ be $\B$-ample and $\Pe$ be a locally free left $\B$-module of rank 1 with $N_\B(\Pe)\cong \sts'$.
\\  If $\Pe^{*_\B}\cong\Pe^\vee$ and $N_\B(\Pe)\cong\sts'$, then a pair $(N, *)$ with $N:\Pe\to\sts'$ a norm on $\Pe$
and an involution $*:\Pe\to \overline{\Pe^\vee}$ on $\Pe$ is called $\A$-{\it admissible} if
 $\langle w,w^*\rangle\in\A$ and $N_\B(\langle w,w^*\rangle)=N(w)N(w)^{*_\B}$
for $w\in\Pe$.\\
 $\Pe$ is called $\A$-{\it admissible} if there is a cubic norm $N:\Pe\to \mathcal{O}_X'$ and a nondegenerate
 $*_\B$-sesquilinear
form $h:\Pe\times\Pe\to \B$ (i.e., $h(aw,bv)=ah(w,v)b^{*_\B}$ and $h$ induces an $\B$-module isomorphism
$j_h:\Pe\to \overline{\Pe^\vee}$) such that $h(w,w)\in\A$ and $N_\B(h(w,w))=N(w)N(w)^{*_\B}$
for $w\in\Pe$.
Note that $\Pe^\vee\cong\Pe^{*_\B}$   and that therefore $j_h$ (denoted $*$ from now on) is an involution  on $\Pe$
such that $\langle w,w^*\rangle\in\A$ and $N_\B(\langle w,w^*\rangle)=N(w)N(w)^{*_\B}$
for $w\in\Pe$.

\begin{theorem} Let $(\A,\sts)$ be $\B$-ample and $\Pe$ be a locally free left $\B$-module of rank 1 with $N_\B(\Pe)\cong \sts'$
which is $\mathcal{A}$-admissible. Define
$$\widetilde{\J}=\A\oplus \Pe,$$
$$\tilde{1}=(1,0)\in H^0(X,\widetilde{\J}),$$
$$\begin{array}{r}
\widetilde{N}(a,w)=N_\B(a)+N(w)+\check{N}(w^*)-T_\B(a,\langle w,w^*\rangle)\\
=N_\B(a)+N(w)+N(w)^{*_\B}-T_\B(a,\langle w,w^*\rangle),
\end{array}$$
$$(a,w)^{\widetilde{\sharp}}=(a^{\sharp_\B}-\langle w,w^*\rangle,w^{*\check{\sharp}} -aw)$$
for $a\in \A$ and $w\in \Pe$. Then $(\widetilde{N},\widetilde{\sharp},\tilde{1})$ is a cubic form with adjoint
and base point on $\widetilde{\J}$ and
$$\widetilde{T}((a,w),(c,v))=T_\B(a,c)+T_\B(\langle w,v^*\rangle)+T_\B(\langle v,w^*\rangle)$$
for $a,c\in \A$ and $v,w\in \Pe$ is the trace.
\end{theorem}
 The induced Jordan algebra $\widetilde{\J}(\widetilde{N},\widetilde{\sharp},\tilde{1})$
is denoted by $\mathcal{J}(\B,\A,\Pe,N,*)$ and called the {\it  Tits process}.
It generalizes the classical Tits process over rings  given in [P-R].

 Let $\J=\mathcal{J}(N_\J,\sharp_\J,1)$ be an Albert algebra over $X$ which contains a
subalgebra of the kind $\h(\B,*_\B)$, where $\B$ is an Azumaya algebra over $\mathcal{O}_X'$ of constant rank 9
together with an involution $*_\B$ and $\sts'$ an $\sts$-algebra of constant rank 2 with $\h(\sts',*_\B)=\sts$.
Then there exists an $\h(\B,*_\B)$-admissible left $\B$-module $\Pe$ of rank 1 and a $\h(\B,*_\B)$-admissible
pair $(N,*)$ such that the canonical embedding $\h(\B,*_\B)  \hookrightarrow \J$ can be extended to
an isomorphism $$\mathcal{J}(\B,\h(\B,*_\B),\Pe,N,*)\to \J.$$

We will not always get all possible cubic Jordan algebras over $X$ using the generalized Tits process,
see [P-S-T2] for examples.

\subsection{The first Tits construction}  (due to Achhammer [Ach1, 2], see [Pu2] for details.)
 Let $\A$ be a unital associative $\sts$-algebra, $(N_\A,\sharp_\A,1)$ a cubic form with
adjoint and base point on $\A$, $\Pe\in {\rm Pic}_l\,\A$ such that $N_\A(\Pe)\cong\sts$ and $N$ a cubic
norm on $\Pe$. Let $\A^+=\mathcal{J}(N_\A,\sharp_\A,1)$ and $N_\A(xy)=N_\A(x)N_\A(y)$ for all $x,y\in\A$. Define
$$\begin{array}{l}
\widetilde{\J}=\A\oplus \Pe\oplus\Pe^\vee,\\
\widetilde{1}=(1,0,0)\in H^0(X,\J),\\
\widetilde{N}(a,w,\check{w})=N_\A(a)+N(w)+\check{N}(\check{w})-T_\A(a,\langle w,\check{w}\rangle)\\
(a,w,\check{w})^{\widetilde{\sharp}}=(a^{\sharp_\A}-\langle w,\check{w}\rangle, \check{w}^{\check{\sharp}}-
aw, w^\sharp-\check{w}a)
\end{array}$$
for $a\in\A$, $w\in\Pe$, $\check{w}\in\Pe^\vee$, then $(\widetilde{N},\widetilde{\sharp},\widetilde{1})$
is a cubic form with adjoint and base point on $\widetilde{\J}$ and has trace form
$$\widetilde{T}((a,w,\check{w}),(c,v,\check{v}))=T_\A(a,c)+T_\A(\langle w, \check{v}\rangle)+T_\A(\langle v,
\check{w}\rangle).$$
The induced Jordan algebra $\mathcal{J}(\widetilde{N},\widetilde{\sharp},\widetilde{1})$ is denoted by $\mathcal{J}(\A,\Pe,N)$
and called a {\it (generalized) first Tits construction} [Ach1, 2].
This first Tits construction generalizes the  {\it classical} first Tits construction from [P-R]. $\A^+$ identifies canonically with a subalgebra of
$\mathcal{J}(\A,\Pe,N)$.\\
If
$$\begin{array}{l} \sts'=\sts\oplus\sts, \, (r,s)^{*_\B}=(s,r),\\
\B=\A\oplus\A^{\rm op}, \, (a,c)^{*_\B}=(c,a),\, 1_\B=(1,1),\\
N_\B(a,c)=(N_\A(a),N_\A(c)), \, (a,c)^{\sharp_\B}=(a^{\sharp_\A},c^{\sharp_\A}),\\
\sts^0=\{(r,r)\,|\,r\in\sts\},\, \A^0=\{(a,a)\,|\,a\in\A\},\\
\Pe^0=\Pe\oplus\Pe^\vee,\\
N^0(w,\check{w})=(N(w),\check{N}(\check{w})),\, (w,\check{w}^{*^0})=(\check{w},w),
\end{array}$$
then
$$\mathcal{J}(\B,\A^0,\Pe^0,N^0,*^0)\cong \mathcal{J}(\A,\Pe,N).$$
For $P\in X$, $$\widetilde{\J}(\A,\Pe,N)_P\cong \widetilde{\J}(\A_P,\Pe_P,N_P)$$
where the right hand side is a  classical first Tits construction over $\mathcal{O}_{P,X}$ in the sense of [P-R].
If $\J=\mathcal{J}(N_\J,\sharp_\J,1)$ is an Albert algebra over $X$ containing a subalgebra of the kind $\A^+$ with $\A$ an
Azumaya algebra over $X$, then there exist suitable $\Pe$ and $N$ such that $\mathcal{J}(\A,\Pe,N)\cong \J$.

The  Tits process can be embedded into a first general first Tits construction, analogously as shown
in [P-R], [M1], cf.  [Ach1, 2.28] or [Pu2].

\section{Curves of genus zero}

Let $X$ be a curve of genus zero; i.e. a geometrically integral, complete, smooth scheme of dimension one over $k$.
Let $P_0\in X$ be a closed point of minimal degree and $\lb(mP_0)$ the  line bundle over $X$ associated with the divisor
$mP_0$. The isomorphism $\mathbb{Z}\cong {\rm Pic}\,X$ is given by the map $m\to \lb(mP_0)$.
If $X$ is rational, $P_0$ has degree 1 and $\lb(mP_0)\cong \mathcal{O}_X(m)$.
In that case let  ${\bf h}(m)$ denote the hyperbolic plane given by the quadratic form
$(a,b)\to ab$ on $ \mathcal{O}_X(m)\oplus \mathcal{O}_X(-m)$.
If $X$ is nonrational, let $D_0$ be the quaternion division algebra
associated to $X$.
 If $k'/k$ is a finite separable field extension which is a maximal subfield
 of $D_0$, then we let $X'=X\times_k k'$ and have $X'\cong\mathbb{P}^1_{k'}$. In that case, let $\mathcal{E}_0
 =tr_{k'/k}( \mathcal{O}_{X'}(1))$
 be the indecomposable vector bundle of rank 2 with $D_0={\rm End}(\mathcal{E}_0)$.

\begin{example} Using Petersson's classification of the composition algebras over $X$ [P], we can list
 all the Jordan  algebras of the kind $\mathcal{H}_3(\mathcal{C},\Gamma)$ with $\mathcal{C}$
a composition algebra or 0 (the first 4 types are defined over $k$):
\begin{enumerate}
\item $\sts^+$, if $k$ has characteristic not 3;
\item $\J=\mathcal{O}_X^+\times\mathcal{O}_X^+\times\mathcal{O}_X^+$;

\item $\mathcal{H}_3(k\times k,\Gamma)\otimes_k \sts$,
 in particular, $\mathcal{H}_3(k\times k)\otimes_k \sts\cong{\rm Mat}_3(k)^+\otimes_k \sts$;

\item $\mathcal{H}_3(C_0,\Gamma_0)\otimes_k \sts$ with $C_0$ a composition division algebra over $k$;
\item $\mathcal{H}_3(\quat,\Gamma)$ with $\quat$ a split quaternion algebra; i.e., the underlying module structure
of $\mathcal{H}_3(\quat,\Gamma)$ is
given by $\mathcal{O}^9_X\oplus  \mathcal{L}(mP_0)^3\oplus \mathcal{L}(-mP_0)^{3}$ with $m\in\mathbb{Z}$ arbitrary;
\item $\mathcal{H}_3(\mathcal{C},\Gamma)$ with $\mathcal{C}$ a split octonion algebra over $k$. Thus the underlying
module structure of $\mathcal{H}_3(\mathcal{C},\Gamma)$ is given by
$$\mathcal{O}^9_X\oplus \lb(m_1P_0)^3\oplus \lb(m_2P_0)^3\oplus  \lb(-(m_1+m_2)P_0)^3
\oplus \lb(-m_1P_0)^3\oplus \lb(-m_2P_0)^3\oplus  \lb((m_1+m_2)P_0)^3 $$
 for some $m_1\geq m_2\geq 0$, $m_1>0$, or, if $X$ is not rational, by
$$\mathcal{O}^9_X\oplus \lb(-(2m+1)P_0)^3\oplus [\lb(mP_0)\otimes\mathcal{E}_0]^3\oplus \lb((2m+1)P_0)^3\oplus
[\lb(-mP_0)\otimes\check{\mathcal{E}_0}]^3,$$
the integer  $m$ being uniquely determined by $\mathcal{C}$;
\item if $X$ is nonrational, $\mathcal{H}_3(\mathcal{C},\Gamma)$ with
$\mathcal{C}\cong {\rm Cay} (\mathcal{D},\mathcal{P}, N_0)$,
where $\mathcal{D}=D_0\otimes \mathcal{O}_X$ and $\mathcal{P}=\mathcal{P}_1\oplus \mathcal{P}_2$ with
$\mathcal{P}_1=\mathcal{L}(mP_0)\otimes \check{{\mathcal{E}}}_0$ and
$\mathcal{P}_2=\mathcal{L}((-m+1)P_0)\otimes \check{\mathcal{E}}_0$ for some integer $m\geq 0$ uniquely
determined by $\mathcal{C}$ [P, 4.9]. Thus the underlying module structure of $\mathcal{H}_3(\mathcal{C},\Gamma)$ is given by
$$\mathcal{O}^{15}_X\oplus  [\mathcal{L}(mP_0)\otimes \check{{\mathcal{E}}}_0]^3\oplus [\mathcal{L}((-m+1)P_0)
\otimes \check{\mathcal{E}}_0]^3 .$$
\end{enumerate}
For the definition of $\mathcal{H}_3(\mathcal{C},\Gamma)$ over $X$, see for instance [Pu2, Example 1].
\end{example}

\begin{remark} (i) (imitating [P, 5.5]) Let $X=\mathbb{P}_k^1$ and let $\sigma:X\to k$ denote the structure morphism of $X$.
 Let $\J$ be a  Jordan algebra of (automatically constant) rank
  with nondegenerate quadratic trace form $T_\J$
over $X$. By Knebusch's Theorem [Kne, Theorem 13.2.2] we can decompose the quadratic space $(\J,T_\J)$ as
$$(\J,T_\J)={\bf h}(m_1)\perp\dots\perp {\bf h}(m_r)\perp \sigma^*(V,q)$$
with $m_1\geq\dots\geq m_r >0$ integers. For the global sections we get
$$J=H^0(X,\J)=H^0(X,\sts(m_1))\perp\dots\perp H^0(X,\sts(m_1))\perp V$$
and $1_J=\sum_{j=1}^rg_j+e$ for $g_j\in k[x_0,x_1]$ homogeneous of degree $m_j$, $e\in V$.
Hence $q(e)=T_\mathcal{J}(1_J)=3$. So $V$ has dimension at least one
if $\J$ has odd rank and at least 2 if it has even rank.
Thus, if $\J$ has rank 3, either $\J$ is globally free as $\sts$-module or
$\J\cong {\bf h}(m)\perp\sts$.
\\(ii) Let $X$ be nonrational. Let $\J$ be as in (i). Let $k'$ be a splitting field of $X$ and let $X'=X\times_kk'$. By (i) applied to
$\J\otimes\mathcal{O}_{X'}$ (or by a direct argument involving module structures
of selfdual modules), if $\J$ has rank 3, it is either globally free
 or $\J\cong \lb(sP_0)\oplus \lb(-sP_0)\oplus \sts$ for some $s>0$.
\end{remark}

\begin{lemma}  (i) If $k$ has characteristic not 2 or 3, every first Tits construction over $X$ starting with $\mathcal{O}_X$
 is defined over $k$.\\
 (ii) Let $l/k$ be a separable  quadratic field extension with ${\rm Gal}(l/k)=\langle \omega\rangle$.
 Let $X_l=X\times_k l$. Every  Tits process over $X$ starting with $\B=\mathcal{O}_{X_l}$ and $*_\B=\omega$
 is defined over $k$.
\end{lemma}

\begin{proof} (i) Every first Tits construction  over $X$ starting with $\mathcal{O}_X$ is of the kind
$\mathcal{J}(\mathcal{O}_X, \mathcal{L}, N_{\mathcal{L}})$ with $\mathcal{L} \in \, _3 {\rm Pic} X$. However,
 ${\rm Pic}\,X\cong\mathbb{Z}$, so we only have $\mathcal{J} (\mathcal{O}_X,\mathcal{O}_X, \mu)\cong
 \mathcal{J}(k,\mu)\otimes_k \sts$, for some $\mu\in k^\times$.\\
 (ii) Let $X=\mathbb{P}_k^1$. If $\Pe\in{\rm Pic}\,X_l$ is $\sts$-admissible then $\Pe=\mathcal{O}_{X_l}$
 (each $\Pe$ is already defined over $X$, so $^\omega\Pe\cong\Pe$ and thus $^\omega\Pe\cong\Pe^\vee$ implies
 $\Pe\cong\Pe^\vee$). Therefore
 $\J=\mathcal{J}(\mathcal{O}_{X_l},\sts,\Pe,N,*)$ is globally free as $\sts$-module implying that it is defined over $k$
  [Pu2, Lemma 1].\\
 If $X$ is nonrational and $\J$ is a  Tits process over $X$ then, given any splitting field $k'$ of the quaternion
 division algebra associated with $X$, $\J\otimes \mathcal{O}_{X'}$ is  a Tits process over
$\mathcal{O}_{X'}$ and so defined over $k'$. Hence $\J$ is globally free as $\sts$-module implying that it is defined over
$k$ [Pu2, Lemma 1].
\end{proof}

We cannot exclude the possibility that there are Jordan algebras $\mathcal{J}(N,\sharp,1)$ of degree 3 over a curve of
 genus zero which do not arise from
a first Tits construction or a Tits process. By Remark 1, the underlying module structure of such algebras
must be $\sts\oplus \lb(mP_0)\oplus\lb(-mP_0)$ ($m>0$).

\subsection{\'Etale first Tits constructions over $X$}

Let $\A$ be a commutative associative $\sts$-algebra of constant rank 3 such that $\A^+=\mathcal{J}(N_\A,\sharp_\A,1)$
and where $\A(P)$ is a cubic
\'etale $k(P)$-algebra for all $P\in X$. The first Tits construction $\J=\mathcal{J}(\A,\Pe,N)$ is called an {\it
\'etale first Tits construction} [Pu2, 4.1].

\begin{lemma} Let $\A=\sts\times\sts\times\sts$. The \'etale first Tits construction
$\mathcal{J}(\A,\Pe,N)$ is a Jordan algebra over $X$ of rank 9 with underlying module structure
$$\sts^3\oplus \lb(m_1P_0)\oplus \lb(m_2P_0)\oplus \lb(-(m_1+m_2)P_0)\oplus \lb(-m_1P_0)\oplus
\lb(-m_2P_0)\oplus \lb((m_1+m_2)P_0).$$
\end{lemma}

\begin{proof} Every left $\A$-module $\Pe$ of rank one with $N_\A(\Pe)\cong\sts$ satisfies
$$\Pe\cong \lb(m_1P_0)\oplus \lb(m_2P_0)\oplus \lb(-(m_1+m_2)P_0)$$
with $m_i\in\mathbb{Z}$ ([Ach1, 4.1], see also [Pu2, Example 8]).
Choose an isomorphism $\alpha:\lb(m_1P_0)\otimes \lb(m_2P_0)\otimes \lb(-(m_1+m_2)P_0)\to\sts$, then $N(x,y,z)=\alpha(x\otimes y\otimes z)$
defines a norm on $\Pe$ and $\mathcal{J}(\A,\Pe,N)$ has the underlying module structure claimed in the assertion.
\end{proof}

\begin{lemma} Let $\A=k'\otimes \sts$ with $k'$ a cubic  \'{e}tale  field extension of $k$.
 Then there exists only the classical  \'etale first Tits construction for $\A$ which yields a Jordan algebra,
which is defined over $k$.
\end{lemma}

\begin{proof} Let $X'=X\times_k k'$.  Identify ${\rm Pic}\,\A={\rm Pic}\,X'$.
 The canonical morphism
 $$G: {\rm Pic}\,X\to {\rm Pic}\,\A,\,\lb(mP_0)\to \lb(mP_0)\otimes \mathcal{O}_{X'}$$
is bijective. Therefore every $\Pe\in {\rm Pic}\,\A$ is of the form
 $\Pe\cong (k'\otimes \sts)\otimes \lb(mP_0)$ and if $N$ is the norm on $k'\otimes_k \sts$, then
$$\sts\cong N(\Pe)\cong N((k'\otimes \sts)\otimes \lb(mP_0))\cong\lb(3mP_0)$$
implies $m=0$. This yields the assertion.
\end{proof}

\begin{lemma} Let $\torus$ be a non-split quadratic \'etale algebra over $X$, that means
$\torus=k'\otimes \sts$ with $k'/k$ a separable quadratic field extension of $k$.  Let $\A=\sts\times\torus$.
Then $\mathcal{J}(\A,\Pe,N)$ is a Jordan algebra over $X$ of rank 9 with underlying module structure\\
 $$\J\cong\sts^3\oplus [ \lb(-2mP_0)\oplus \lb(mP_0)\oplus\lb(mP_0)]\oplus  [\lb(2mP_0)\oplus \lb(-mP_0)\oplus\lb(-
mP_0)]$$  or, if $X$ is nonrational and $k'$ a splitting field of $X$,
$$\J\cong\sts^3\oplus   [\lb(-(2m+1)P_0)\oplus tr_{k'/k}(\mathcal{O}_{X'}(2m+1))] \oplus
[\lb(2m+1P_0)\oplus tr_{k'/k}(\mathcal{O}_{X'}(-2m-1))].$$
\end{lemma}

\begin{proof}  Let ${\rm Gal}(k'/k)=\langle \sigma\rangle$.
We know that ${\rm Pic}\,\torus\cong\mathbb{Z}$. Obviously, each $\A=\sts\times \torus$
 is defined over $k$, hence
$\A^+=\mathcal{J}(N_\A,\sharp_\A,1)$. Let $\J=\mathcal{J}(\A,\Pe,N)$.
By [Pu2, Lemma 5], every left $\A$-module $\Pe$ of rank one with $N_\A(\Pe)\cong\sts$ is either a direct sum of line bundles
$\lb_0,\M_0,\s_0$ over $X$ with $\lb_0\otimes\M_0\otimes\s_0\cong\mathcal{O}_{X}$; i.e.,
 $$\Pe\cong \lb_0\oplus\M_0\oplus \s_0,$$
  or there is a line bundle $\s_0$ over $X$ and a line bundle $\lb$  over $X'$ which is not defined over $X$
with $\lb\otimes\, ^\sigma\lb\otimes\s_0\cong\mathcal{O}_{X'}$ and
$$\Pe\cong tr_{k'/k}(\lb)\oplus\s_0.$$
This last case can only happen if $X$ is nonrational and $k'$ is a splitting field of the quaternion division
algebra $D_0$ associated to $X$.

Since $\check{H}^1(X,\A^\times)={\rm Pic}\,X\oplus {\rm Pic}\,\torus$, every left $\A$-module $\Pe$
of rank one satisfies $\Pe\cong \lb\oplus\F$ with $\F\in {\rm Pic}\torus$ and an invertible $\sts$-module $\lb$.
 We have $ N_\A(\Pe)\cong \sts$ iff $\lb\otimes N_\torus(\F)\cong \sts$ iff $\lb^\vee\cong N_\torus(\F)$.
For $\M\in {\rm Pic}\,X$, $\M\otimes\torus\in {\rm Pic}\,\torus$ has $N_\torus(\M\otimes\torus)\cong \M
^2$.

Now ${\rm Pic}\torus$ can  be identified with  ${\rm Pic}X'$, and unless $X$ is nonrational and
 $k'$ is a splitting field of $X$, the canonical morphism
 $$G: {\rm Pic}X'\to {\rm Pic}X', \lb\to \lb\otimes \mathcal{O}_{X'}$$
 is bijective, implying that each $\F\in{\rm Pic}\torus$ is of the form
  $\F\cong (k'\otimes \sts)\otimes \lb(mP_0)$. If $N$ is the norm on $k'\otimes_k \sts$, then
$$ N(\F)\cong N((k'\otimes \sts)\otimes \lb(mP_0))\cong\lb(2mP_0),$$
 Hence
$$\Pe\cong \lb(-2mP_0)\oplus \lb(mP_0)\oplus\lb(mP_0)$$
as $\sts$-module for some $m\in\mathbb{Z}$ and
$$\J=\sts^3\oplus  \lb(-2mP_0)\oplus \lb(mP_0)\oplus\lb(mP_0)\oplus  \lb(2mP_0)\oplus \lb(-mP_0)\oplus\lb(-
mP_0)$$
as $\sts$-module.

 If $X$ is nonrational and $k'$ is a splitting field of $D_0$, then $G$ is not bijective and it is possible that
the indecomposable
$$\F=\lb(mP_0)\otimes tr_{k'/k}(\mathcal{O}_{X'}(1))\cong tr_{k'/k}(\mathcal{O}_{X'}(2m+1))$$
($m\in \mathbb{Z}$ arbitrary) is an element in ${\rm Pic}\torus$. Then
$$\F\otimes \mathcal{O}_{X'}\cong \mathcal{O}_{X'}(2m+1)\oplus \mathcal{O}_{X'}(2m+1).$$
Since $\mathcal{O}_{X'}(2m+1)\otimes \mathcal{O}_{X'}(2m+1)\cong\mathcal{O}_{X'}(4m+2)$ this yields
$$\Pe\cong \lb(-(2m+1)P_0)\oplus tr_{k'/k}(\mathcal{O}_{X'}(2m+1)).$$
\end{proof}

It is not clear whether this last case actually occurs.

\subsection{First Tits constructions starting with algebras of rank 5}

\begin{proposition} Let $\D$ be a non-split quaternion algebra over $X$.  Let $\A=\sts\times\D$.
Then   a first Tits construction $\J=\mathcal{J}(\A,\Pe,N)$ has the following underlying module structure:\\
(i) If $X$ is nonrational and $D\cong D_0$, then
$$\J\cong\sts^5\oplus [ \lb(2mP_0) \oplus\lb(mP_0)^4]\oplus [ \lb(-2mP_0) \oplus\lb(-mP_0)^4] $$
or
$$\begin{array}{l}
\J\cong \sts^5\oplus [\lb((m_1+m_2-1)P_0) \oplus tr_{k'/k}(\mathcal{O}_{X'}(-2m_1+1))
 \oplus tr_{k'/k}(\mathcal{O}_{X'}(-2m_2+1))]\\
\oplus [ \lb(-(m_1+m_2-1)P_0) \oplus tr_{k'/k}(\mathcal{O}_{X'}(2m_1-1)) \oplus tr_{k'/k}(\mathcal{O}_{X'}(2m_2-1)) ].
\end{array}$$
\noindent (ii) If $D$ is a quaternion division algebra over $k$ and $D\not\cong D_0$ if $X$ is nonrational,
then
$$\J\cong\sts^5\oplus [\lb(-4mP_0)\oplus \lb(mP_0)^4]\oplus [\lb(4mP_0)\oplus  \lb(-mP_0)^4]. $$
\end{proposition}

\begin{proof}  We know that $\D=D\otimes \sts$ with $D$ a quaternion division algebra over $k$ [P], hence $\A$ is defined over $k$.
 Every left $\A$-module $\Pe$
of rank one satisfies $\Pe\cong \lb\oplus\Pe_0$ with $\Pe_0\in {\rm Pic}_l\,\D$ and an invertible $\sts$-module $\lb$.
We have $ N_\A(\Pe)\cong \sts$  iff $\lb^\vee\cong N_\D(\Pe_0)$ [Pu2, Lemma 6].
\\ (i) Let $D\cong D_0$, then $\D=D_0\otimes\sts\cong \E nd_X(\E_0)$. Each $\Pe_0\in {\rm Pic}_l\,\D$ is of the form
$\Pe_0\cong \F^\vee\otimes\E_0$
 and $ N_\D(\Pe_0)\cong {\rm det}\,\F^\vee\otimes \lb(P_0)$, where $\F$ is a locally free $\sts$-module of constant
rank 2. We have either
$$\F=\lb(m P_0)\otimes \E_0 \text{ or } \F=\lb(m_1P_0)\oplus \lb(m_2P_0)$$
 for arbitrary $m,m_1,m_2\in \mathbb{Z}$. Therefore either
$$\Pe_0\cong \lb(-mP_0)\otimes \E_0^\vee\otimes\E_0\cong \lb(mP_0)^4$$
or
$$\Pe_0\cong \lb(-m_1P_0)\otimes\E_0\oplus \lb(-m_2P_0)\otimes\E_0$$
as $\sts$-module. We obtain
$$ N_\D(\Pe_0)\cong  \lb(-(2m+1)P_0) \otimes \lb(P_0)=\lb(-2mP_0)$$
or
$$ N_\D(\Pe_0) \cong \lb(-(m_1+m_2-1)P_0).$$
This implies
 $$\Pe\cong   \lb(2mP_0) \oplus\lb(mP_0)^4$$
or
$$\begin{array}{l}
\Pe\cong  \lb((m_1+m_2-1)P_0) \oplus \lb(-m_1P_0)\otimes\E_0\oplus \lb(-m_2P_0)\otimes\E_0\\
\cong
\lb((m_1+m_2-1)P_0) \oplus tr_{k'/k}(\mathcal{O}_{X'}(-2m_1+1)) \oplus tr_{k'/k}(\mathcal{O}_{X'}(-2m_2+1)).
\end{array}$$
\noindent (ii) For every $\M\in {\rm Pic}\,X$, $\Pe_0\in {\rm Pic}_l\,\D$ we have $\M\otimes\Pe_0\in {\rm Pic}_l\,\D$.
We show that each element in ${\rm Pic}\,\D$ is of this type:
By [P, 4.5], the only right $\D$-module of rank one of norm one is $\D$ itself. Write $\D_\D$ for $\D$ viewed as a right
$\D$-module.
Suppose that ${\rm Pic}_r\,\D$ is non-trivial, then there is an element $\Pe_0\not=\D_\D$ such that
$N_\D(\Pe_0)=\lb(nP_0)$ for some $n\not=0$ and
$$N_\D(\lb(mP_0)\otimes\Pe_0)\cong \lb(2mP_0)\otimes N_\D(\Pe_0)=\lb(2mP_0)\otimes\lb(nP_0)=\lb((2m+n)P_0).$$
Choose $2m=-n$ to get norm one, then $\lb(mP_0)\otimes\Pe_0$ would be a non-trivial element of norm one, a contradiction.
(If $X$ is rational, there is an analogous argument.) Thus ${\rm Pic}_r\,\D$ is trivial, and hence
 so is ${\rm Pic}_l\,\D$.
Therefore $$\Pe\cong \lb(-4mP_0)\oplus \lb(mP_0)\otimes \D .$$
\end{proof}

\begin{proposition} Let $\D$ be a split quaternion algebra over $X$. Then $\D=\E nd_X(\sts\oplus \lb(mP_0))$.
 Let $\A=\sts\times\D$. Then a first Tits construction $\J=\mathcal{J}(\A,\Pe,N)$ has
one of  the following underlying module structures:\\
$$\begin{array} {l}
\J\cong\sts^3\oplus \lb(mP_0)\oplus \lb(-mP_0)  \\ \oplus
 [ \lb((2m_1-m+1)P_0)\oplus \lb(-m_1P_0)\otimes \E_0^\vee\oplus \lb((-m_1+m)P_0)\otimes \E_0^\vee ]
\\ \oplus
  [\lb(-(2m_1-m+1)P_0)\oplus \lb(m_1P_0)\otimes \E_0\oplus \lb((m_1-m)P_0)\otimes \E_0 ]
\end{array}$$
or
$$\begin{array}{l}
\J\cong \sts^3\oplus \lb(mP_0)\oplus\lb(-mP_0) \\ \oplus
[ \lb((m_2+m_3-m)P_0)\oplus \lb(-m_2 P_0)\oplus \lb(-m_3 P_0)\oplus \lb((-m_2+m)P_0)\oplus \lb((-m_3+m)P_0)
]\\
\oplus [ \lb(-(m_2+m_3-m)P_0)\oplus  \lb(m_2 P_0)\oplus \lb(m_3 P_0)\oplus \lb((m_2-m)P_0)\oplus \lb((m_3-m)P_0) ].
\end{array}$$
\end{proposition}

\begin{proof}  If $\D$ is a split quaternion algebra over $X$ then $\D=\E nd_X(\sts\oplus \lb(mP_0))$, $m\in\mathbb{Z}$.
Every left $\A$-module $\Pe$
of rank one satisfies $\Pe\cong \lb\oplus\Pe_0$ with $\Pe_0\in {\rm Pic}_l\,\D$ and an invertible $\sts$-module $\lb$,
and $ N_\A(\Pe)\cong \sts$ iff $\lb^\vee\cong N_\D(\Pe_0)$ [Pu2, Lemma 6].
 Each $\Pe_0\in {\rm Pic}_l\,\D$ is of the form
 $\Pe_0\cong  \F^\vee\otimes[\sts\oplus\lb(mP_0)]\cong \F^\vee\oplus
 \F^\vee\otimes\lb(mP_0)$
 and $ N_\D(\Pe_0)\cong {\rm det}\,\F^\vee \otimes \lb(mP_0)$, where $\F$ is a locally free $\sts$-module of constant
rank 2. We have either
$$\F=\lb(m_1P_0)\otimes \E_0 \text{ or } \F=\lb(m_2P_0)\oplus \lb(m_3P_0)$$
 for arbitrary $m_1,m_2,m_3\in \mathbb{Z}$. Hence
$$\Pe_0\cong\lb(-m_1P_0)\otimes \E_0^\vee\oplus \lb((-m_1+m)P_0)\otimes \E_0^\vee$$
or
$$\Pe_0\cong \lb(-m_2P_0)\oplus \lb(-m_3P_0)\oplus \lb((-m_2+m)P_0)\oplus \lb((-m_3+m)P_0).$$
We obtain
$$ N_\D(\Pe_0)\cong \lb(-(2m_1+1)P_0)\otimes   \lb(mP_0)=\lb((-2m_1+m-1)P_0)$$
or
$$ N_\D(\Pe_0)\cong  \lb(-(m_2+m_3)P_0)\otimes   \lb(mP_0)=\lb((-m_2-m_3+m)P_0).$$
Therefore
$$\Pe\cong \lb((2m_1-m+1)P_0)\oplus [\lb(-m_1P_0)\otimes \E_0^\vee\oplus \lb((-m_1+m)P_0)\otimes \E_0^\vee] $$
or
$$\Pe\cong  \lb((m_2+m_3-m)P_0)\oplus [ \lb(-m_2 P_0)\oplus \lb(-m_3 P_0)\oplus \lb((-m_2+m)P_0)\oplus \lb((-m_3+m)P_0) ].$$
\end{proof}

\section{Albert algebras over curves of genus zero which are first Tits constructions}

We keep the notations of Section 3.
\begin{example} (a) ([Ach1, 4.4]) Let $X$ be nonrational. For every locally free $\sts$-module $\mathcal{E}$
 of constant rank 3, $\A=\E nd_X(\E)$ is an Azumaya algebra of rank 9.
 Let $\F$ be another locally free $\sts$-module of constant rank 3 such there exists an isomorphism
$\alpha: {\rm det}\, \E\to {\rm det}\, \F$. We have the following possibilities for
 $\E$ and $\F$:
$$ \begin{array}{l}
\E=\lb(m_1P_0) \otimes\E_0\oplus\lb(m_2P_0) \text{ and }\\
\F=\lb(s_1P_0) \otimes\E_0\oplus\lb(s_2P_0) \text{ or }\\
\F=\lb(l_1P_0)\oplus\lb(l_2P_0)\oplus\lb(l_3P_0).
 \end{array}$$
 with $2m_1+m_2=2s_1+s_2$ and $2m_1+m_2=l_1+l_2+l_3$, or
 $$ \begin{array}{l}
\E=\lb(n_1P_0)\oplus\lb(n_2P_0)\oplus\lb(n_3P_0) \text{ and }\\
\F=\lb(s_1P_0) \otimes\E_0\oplus\lb(s_2P_0)  \text{ or }\\
\F=\lb(l_1P_0)\oplus\lb(l_2P_0)\oplus\lb(l_3P_0).
 \end{array}$$
 with $n_1+n_2+n_3=2s_1+1+s_2$ and $n_1+n_2+n_2=l_1+l_2+l_3$. Then the Albert algebra $\mathcal{J}(\A,\Pe,N)$ over $X$,
  $\Pe=\h om_X(\F,\E)$, is  isomorphic to
\\(i)\begin{align*}
\left [\begin {array}{ccc}
\E nd(\E_0)&\lb(-aP_0)\otimes\E_0\\
\lb(aP_0)\otimes\E_0^\vee&\sts\\
\end {array}\right ]  &
\oplus
\left [\begin {array}{ccc}
\lb(bP_0)\otimes\E nd(\E_0)&\lb((-2b-a)P_0)\otimes\E_0\\
\lb((a+b)P_0)\otimes\E^\vee& \lb(-2bP_0)  \\
\end {array}\right ]
\\ &
\oplus
\left [\begin {array}{ccc}
\lb(-bP_0)\otimes\E nd(\E_0)&\lb((-a-b)P_0)\otimes\E_0\\
\lb((a+2b)P_0)\otimes\E_0^\vee& \lb(2bP_0)  \\
\end {array}\right ]
\end{align*}
with $a=m_2-m_1$, $b=m_1-s_1$;
\\
(ii)
\begin{align*}
& \left [\begin {array}{cc}
\E nd(\E_0)&\lb(-aP_0)\otimes\E_0\\
\lb(aP_0)\otimes\E_0&\sts\\
\end {array}\right ]
\\ & \oplus
\left [\begin {array}{ccc}
\lb(bP_0)\otimes\E_0 & \lb(cP_0)\otimes\E_0& \lb(-(a+b+c+1)P_0)\otimes\E_0 \\
\lb((a+b)P_0)& \lb((a+c)P_0) & \lb(-(b+c+1)P_0)  \\
\end {array}\right ]
\\ & \oplus
\left [\begin {array}{cc}
\lb(-bP_0)\otimes\E_0^\vee & \lb(-(a+b)P_0)\\
\lb(-cP_0)\otimes\E_0^\vee &  \lb(-(a+c)P_0)  \\
\lb((a+b+c+1)P_0)\otimes\E_0^\vee &  \lb((b+c+1)P_0)  \\
\end {array}\right ]
\end{align*}
with $a=m_2-m_1$, $b=m_1-l_1$, $c=m_1-l_2$;
\\(iii)
 \begin{align*}
 & \left [\begin {array}{ccc}
\sts & \lb(aP_0) & \lb(bP_0)\\
\lb(-aP_0) &  \sts & \lb((b-a)P_0)  \\
\lb(-bP_0) & \lb((a-b)P_0) &  \sts \\
\end {array}\right ]
 \\&
  \oplus
\left [\begin {array}{ccc}
\lb(cP_0)\otimes\E_0^\vee & \lb((a+b-2c+1)P_0)\\
\lb((c-a)P_0)\otimes\E_0^\vee &  \lb((b-2c+1)P_0)  \\
\lb((c-b)P_0)\otimes\E_0^\vee  &  \lb((a-2c+1)P_0)  \\
\end {array}\right ]
\\ &
\oplus
\left [\begin {array}{ccc}
\lb(-cP_0)\otimes {\E_0} & \lb((a-c)P_0)\otimes {\E_0} &  \lb((b-c)P_0)\otimes {\E_0} \\
 \lb((2c-a-b-1)P_0)  &  \lb((2c-b-1)P_0) &  \lb((2c-a-1)P_0)\\
\end {array}\right ]
\end{align*}
with $a=n_2-n_1$, $b=n_1-n_3$, $c=n_1-s_1$;
\\(iv)  \begin{align*}
 &\left [\begin {array}{ccc}
\sts & \lb(aP_0) & \lb(bP_0)\\
\lb(-aP_0) &  \sts & \lb((b-a)P_0)  \\
\lb(-bP_0) & \lb((a-b)P_0) &  \sts \\
\end {array}\right ]
\\ & \oplus
\left [\begin {array}{ccc}
\lb(cP_0) & \lb((a+d)P_0)  & \lb((b-c-d)P_0)\\
\lb((c-a)P_0) & \lb(dP_0) & \lb((b-a-c-d)P_0)  \\
\lb((c-b)P_0)  &  \lb((a-b+d)P_0)  &  \lb((-c-d)P_0)  \\
\end {array}\right ]
\\ & \oplus
\left [\begin {array}{ccc}
\lb(-cP_0) & \lb((a-c)P_0) &  \lb((b-c)P_0) \\
  \lb((-a-d)P_0)  &  \lb(-dP_0) &  \lb((b-a-d)P_0) \\
  \lb((-b+c+d)P_0)  &  \lb((a-b+c+d)P_0) &  \lb((c+d)P_0) \\
\end {array}\right ]
 \end{align*}
with $a=n_1-n_2$, $b=n_1-n_3$, $c=n_1-l_1$ and $d=n_2-l_2$.
\\ (b) ([Ach1, 4.5]) Let $X$ be rational.  Let $\E$, $\F$ be two locally free $\sts$-modules of constant rank 3
such there exists an isomorphism $\alpha: {\rm det}\, \E\to {\rm det}\, \F$. Then
$$ \begin{array}{l}
\E=\lb(n_1)\oplus \lb(n_2)\oplus \lb(n_3) \text{ and}\\
\F=\lb(l_1)\oplus\lb(l_2)\oplus\lb(l_3)
 \end{array}$$
 with $n_i,l_i\in\mathbb{Z}$ such that $n_1+n_2+n_3=l_1+l_2+l_3$.
 Put $\A=\E nd_X(\E)$, $a=n_1-n_2$, $b=n_1-n_3$, $c=n_1-l_1$ and $d=n_2-l_2$, then
$\mathcal{J}(\A,\Pe,N)$ with $\Pe=\h om_X(\F,\E)$ is an Albert algebra over $X$ as described in (a) (iv).
\end{example}

\begin{theorem} ([Ach1, 4.7])   Let $k$ be a field of characteristic zero and let $X=\mathbb{P}^1_k$.
Let  $\J$ be an Albert algebra over $X$
which contains $\A^+$ as a subalgebra, where $\A$ is an Azumaya algebra of constant rank 9 over $X$. Then $\J$ is defined
over $k$ or $\J=\mathcal{J}(\A,\Pe,N)$ as in Example 2 (b).
\end{theorem}

\begin{proof} We know that either $\A\cong \E nd_X(\lb(n_1P_0)\oplus\lb(n_2P_0)\oplus\lb(n_3P_0))$ or
$\A\cong A\otimes_k\sts$ with $A$ a central simple
division algebra over $k$ of degree 3 [Kn2]. We still have to show that $\J$ is defined over $k$ in the second case:
Then $\Pe\cong (A\otimes \sts)\otimes \sts(m)$ for some $m\in\mathbb{Z}$ [Kn1, VII.(3.1.)].
If $N$ is the norm on $A\otimes_k \sts$, then
$$\sts\cong N(\Pe)\cong N((A\otimes \sts)\otimes \sts(m))\cong\sts(3m)$$
implies $m=0$, hence the assertion.
\end{proof}

\begin{corollary} Let $X$ be a  nonrational curve over a field $k$ of characteristic zero.
Let $\J$ be an Albert algebra over $X$
which contains $\A^+$ as a subalgebra, where $\A=A\otimes_k \sts$ with $A$ a central simple
division algebra over $k$ of degree 3.
 Then every first Tits construction $\J=\mathcal{J}(\A,\Pe,N)$ is defined over $k$.
\end{corollary}

\begin{proof} Let $k'$ be a splitting field of $D_0$, put $X'=X\times_k k'$.
 Then $X'$ is rational and $\J\otimes \mathcal{O}_{X'}=\mathcal{J}(\A,\Pe,N)$ is defined over $k$ by the proof of Theorem 2.
  Hence $\J$ is globally free as $\sts$-module and thus defined over $k$ [Pu2, Lemma 2].
\end{proof}

For arbitrary base fields, we obtain the same result with a more tedious proof:

\begin{theorem} Let $X$ be  nonrational and let  $\J$ be an Albert algebra over $X$
which contains $\A^+$ as a subalgebra, where $\A=A\otimes_k \sts$ with $A$ a central simple division
 algebra over $k$ of degree 3. Then every first Tits construction
$\J=\mathcal{J}(\A,\Pe,N)$ is defined over $k$.
\end{theorem}

\begin{proof} Let $l$ be a cubic splitting field of the algebra $A$, put $Y=X\times_k l$.
 Then $Y$ is still nonrational. For $\J=\mathcal{J}(\A,\Pe,N)$ we have
  $\J\otimes\mathcal{O}_{Y}= \mathcal{J}(\A\otimes \mathcal{O}_{Y},\Pe\otimes \mathcal{O}_{Y},N\otimes \mathcal{O}_{Y})$
  and $\A\otimes \mathcal{O}_{Y}\cong {\rm Mat}_3(\mathcal{O}_{Y})$. Now
 $ \Pe\otimes \mathcal{O}_{Y}$ is isomorphic to

 \smallskip
\[
\left [\begin {array}{ccc}
\lb(-s_1P_0)\otimes\E_0^\vee & \lb((2s_1+1)P_0)\\
\lb(-s_1P_0)\otimes\E_0^\vee &  \lb((2s_1+1)P_0)  \\
\lb(-s_1P_0)\otimes\E_0^\vee  &  \lb((2s_1+1)P_0)  \\
\end {array}\right ]
\]

\smallskip or
\[
\left [\begin {array}{ccc}
\lb(-l_1P_0) & \lb(-l_2P_0)  & \lb((l_1+l_2)P_0)\\

\lb(-l_1P_0) & \lb(-l_2P_0) & \lb((l_1+l_2)P_0)  \\

\lb(-l_1P_0)  &  \lb(-l_2P_0)  &  \lb((l_1+l_2)P_0)  \\
\end {array}\right ],
\]
hence
$$\Pe\cong (\lb(-s_1P_0)\otimes\E_0^\vee)^3 \oplus \lb((2s_1+1)P_0)^3$$
or
$$\Pe\cong \lb(-l_1P_0)^3\oplus \lb(-l_2P_0)^3 \oplus \lb((l_1+l_2)P_0)^3$$
as $\sts$-module.
Suppose the first. The left $\A$-module structure of $\Pe$ amounts to a homomorphism $\A\to \E nd_X(\Pe)$
of $\sts$-algebras. We have

\smallskip
\[ \E nd_X(\Pe)\cong
\left [\begin {array}{ccc}
{\rm Mat}_3(\E nd_X(\E_0)) &  {\rm Mat}_3(\lb(-(3s_1+1)P_0)\otimes \E_0^\vee) \\
{\rm Mat}_3(\lb((3s_1+1)P_0)\otimes \E_0) &  {\rm Mat}_3(\sts) \\
\end {array}\right ].
\]

\smallskip
Now $(\lb((3s_1+1)P_0)\otimes \E_0)\otimes \mathcal{O}_{X'}\cong   \mathcal{O}_{X'}((3s_1+1)2+1)^2$ for $X'=X\otimes_k k'$,
$k'$ a quadratic splitting field of $D_0$, so if $(3s_1+1)2+1=6s_1+3<0$ then
$$H^0(X,\lb((3s_1+1)P_0)\otimes \E_0)=0$$ and
by passing to the global sections, the above homomorphism induces a homomorphism of $k$-algebras of the type
 \[\varphi:A\to
\left [\begin {array}{ccc}
A &  {\rm Mat}_3(*) \\
0 &  {\rm Mat}_3(k) \\
\end {array}\right ].
\]
Following $\varphi$ with the projection to the lower right-hand corner of the block-matrices yields a homomorphism
$A\to  {\rm Mat}_3(k)$, a contradiction. The same argument works for $6s_1+3>0$. Thus $6s_1+3=0$ which is a contradiction, because
$s_1$ is an integer. We conclude that $\Pe$ can never have this form. Thus
$$\Pe\cong \lb(-l_1P_0)^3\oplus \lb(-l_2P_0)^3 \oplus \lb((l_1+l_2)P_0)^3$$
as $\sts$-module. The left $\A$-module structure of $\Pe$ again amounts to a homomorphism $\A\to \E nd_X(\Pe)$
of $\sts$-algebras. We have

\smallskip
\[\E nd_X(\Pe)\cong
\left [\begin {array}{ccc}
 {\rm Mat}_3(\sts)  &  {\rm Mat}_3(\lb(-l_1+l_2)P_0) &   {\rm Mat}_3(\lb(-(2l_1+l_2)P_0) \\
{\rm Mat}_3(\lb(l_1-l_2)P_0) &  {\rm Mat}_3(\sts)  &  {\rm Mat}_3(\lb(-(2l_2+l_1)P_0) \\
 {\rm Mat}_3(\lb((2l_1+l_2)P_0)  &  {\rm Mat}_3(\lb((2l_2+l_1)P_0)  & {\rm Mat}_3(\sts) \\
\end {array}\right ].
\]

Say, $l_1-l_2<0$, then

\begin{align*}
& H^0(X,\E nd_X(\Pe))\cong
\\ &
\left [\begin {array}{ccc}
 {\rm Mat}_3(k)  &  H^0(X,{\rm Mat}_3(\lb(-l_1+l_2)P_0)) &   H^0({\rm Mat}_3(\lb(-(2l_1+l_2)P_0)) \\
0 &  {\rm Mat}_3(k)  &  H^0({\rm Mat}_3(\lb(-(2l_2+l_1)P_0)) \\
 H^0({\rm Mat}_3(\lb((2l_1+l_2)P_0))  &  H^0({\rm Mat}_3(\lb((2l_2+l_1)P_0))  & {\rm Mat}_3(k) \\
\end {array}\right ].
\end{align*}
Following $\varphi$ with the projection to the upper left-hand corner of the block-matrices yields a homomorphism
$A\to  {\rm Mat}_3(k)$, a contradiction. The same argument works for $l_1-l_2>0$. Thus $l_1-l_2=0$, and
\begin{align*}
& H^0(X,\E nd_X(\Pe))\cong
\\ &
\left [\begin {array}{ccc}
 {\rm Mat}_3(k)  &   {\rm Mat}_3(k) &   H^0(X,{\rm Mat}_3(\lb(-3l_1P_0)) \\
{\rm Mat}_3(k) &  {\rm Mat}_3(k)  &  H^0({\rm Mat}_3(\lb(-3l_1P_0)) \\
 H^0(X,{\rm Mat}_3(\lb(3l_1P_0))  &  H^0(X,{\rm Mat}_3(\lb(3l_1P_0))  & {\rm Mat}_3(k) \\
\end {array}\right ].
\end{align*}
We conclude that
 $$\Pe\cong \lb(-l_1P_0)^3\oplus \lb(-l_1P_0)^3 \oplus \lb(2l_1P_0)^3.$$
 Try to assume that  $l_1<0$, then $H^0(X,{\rm Mat}_3(\lb(3l_1P_0))=0$ and
\begin{align*} &
 H^0(X,\E nd_X(\Pe))\cong
 \\ &
\left [\begin {array}{ccc}
 {\rm Mat}_3(k)  &   {\rm Mat}_3(k) &   H^0(X,{\rm Mat}_3(\lb(-3l_1P_0)) \\
 {\rm Mat}_3(k) &  {\rm Mat}_3(k)  &  H^0({\rm Mat}_3(\lb(-3l_1P_0)) \\
 H^0(X,{\rm Mat}_3(\lb(3l_1P_0))  &  0  & {\rm Mat}_3(k) \\
\end {array}\right ].
\end{align*}
The same argument yields a contradiction. The same works with $l_1>0$, so $l_1=0$ and we get
 $$\Pe\cong \sts^9$$
 is globally free. Hence $\J$ is defined over $k$.
\end{proof}

\begin{corollary}  Let $X$ be a  nonrational curve over a field $k$ of characteristic zero.
Let $\J$ be an Albert algebra over $X$ which contains $\A^+$ as a subalgebra, where $\A=A\otimes\sts$
 is an Azumaya algebra of rank 9 over $X$ which is defined over $k$. Then $\J$ is defined over $k$
  or $\J=\mathcal{J}(\A,\Pe,N)$ as in Example 2 (a).
\end{corollary}

\begin{corollary} Let $X$ be a  nonrational curve over a field $k$ of characteristic zero.
Let $\J$ be an Albert algebra over $X$
which contains $\A^+$ as a subalgebra, where $\A$ is an Azumaya algebra of constant rank 9 over $X$
 such that $\A(\xi)$ is a division algebra.
Then $\J$ is defined over $k$.
\end{corollary}

\begin{proof} Let $k'$ be a quadratic splitting field of the quaternion algebra associated to $X$, put $X'=X\times_k k'$.
Let $K$ be the function field of $X$ and $K'=K\otimes_k k'$ the one of $X'$.
If $\J$ is not defined over $k$, then $\A$ is not defined over $k$ (Theorem 3) and
 $\A\otimes \mathcal{O}_{X'}\cong \E nd_X(\E')$ with $\E'$ not globally free.
Our assumption that
 $\A(\xi)$ is a division algebra excludes the case that $\A\cong \E nd_X(\E)$  for some $\E$ of rank 3.
Thus $\A\not\cong \E nd_X(\E)$ and
 $\A\otimes \mathcal{O}_{X'}\cong \E nd_X(\E')$ implies that  $\A(\xi)\otimes_{K} K'\cong {\rm Mat}_3(K')$.
 Since $K'$ is a quadratic field extension of $K$ and hence cannot be a splitting field of the central simple
 division algebra $\A(\xi)$, this is a contradiction.
\end{proof}

It does not seem clear if, for a nonrational curve, there are Azumaya algebras of constant rank 9 other than those defined over $k$ or those of the
type considered in Example 2 (a). If not, the above results would completely classify the Albert algebras
over a nonrational curve over a field of characteristic zero which are first Tits constructions.
The underlying module structure of the Albert algebra $\mathcal{H}_3(\mathcal{C},\Gamma)$ in Example 1 (7)
is not covered in the cases of Example 2 (a). Thus it cannot be a first Tits construction starting with one of the Azumaya
algebras we considered so far. It might be possible to obtain it as a Tits process.

\section{Albert algebras over curves of genus zero which are obtained by the Tits process}

Unless otherwise specified, we keep the notations of Section 3.
\subsection{}  Let $X=\mathbb{P}_k^1$ and $k$ be a field of characteristic 0.
 Let $K$ be a quadratic field extension of $k$ with ${\rm Gal}(K/k)=\langle\omega\rangle$ and put $X'=X\times_k K$.
To simplify notation, we write $\sts(m)$ instead of $\lb(mP_0)$.

\begin{theorem} Let $(B,*)$ be a central simple associative algebra of degree 3 over $k$
with involution of the second kind. Suppose that $K={\rm Cent}\,B$ and that $B$ is a division algebra over $K$.
Let $*_\B=*\otimes \mathcal{O}_{X'}$.

Suppose $\J$ is an Albert algebra over $X$ containing $\h(\B,*_\B)\cong \h(B,*)\otimes_k\sts$ as a subalgebra.
Then $\J\cong\mathcal{J}(\B, \h (\B,*_\B),\Pe,N,*)$ is defined over $k$.
\end{theorem}

\begin{proof} With $(B,*)$  a central simple algebra of degree 3 over $k$
with involution of the second kind [KMRT, p.~20], $\B=B\otimes_K\mathcal{O}_{X'}$ is an Azumaya algebra of rank 9 over $X'$.
The Albert algebra $\J\otimes \mathcal{O}_{X'}$ contains the subalgebra
$\B^+\cong B^+\otimes \mathcal{O}_{X'}$ [Pu2, Lemma 4], thus is a first Tits construction
starting with $\B^+$ [Pu2, Proposition 3] and therefore defined over $K$ (Theorem 2). As such it must be globally
free as $\mathcal{O}_{X'}$-module.
Therefore $\J$ is globally free as $\sts$-module and thus defined over $k$ [Pu2, Lemma 1].
\end{proof}

\begin{lemma}
If the vector bundle
$$\E=\mathcal{O}_{X'}(m_1)\oplus\mathcal{O}_{X'}(m_2)\oplus\mathcal{O}_{X'}(m_3)$$
admits a regular hermitian form $h$, then
$$(\E,h)\cong\langle a\rangle\oplus(\mathcal{O}_{X'}(m)\oplus\mathcal{O}_{X'}(-m),\mathbb{H}),$$
or $\E$ is globally free and $(\E,h)\cong\langle a,b,c\rangle$ with $a,b,c\in K^\times$.
\end{lemma}

\begin{proof} Since $\E\cong\,^\omega\E^\vee$, the Theorem of Krull-Schmidt implies that either
$\mathcal{O}_{X'}(m_i)\cong \, ^\omega\mathcal{O}_{X'}(m_i)^\vee$ for all $i$ or that $\mathcal{O}_{X'}(m_1)\cong \, ^\omega
\mathcal{O}_{X'}(m_1)^\vee$
and $\mathcal{O}_{X'}(m_2)\cong^\omega\mathcal{O}_{X'}(m_3)^\vee$. Now
 $\mathcal{O}_{X'}(m)\cong \, ^\omega\mathcal{O}_{X'}(m)$ for all $m\in\mathbb{Z}$ since the bundle
$\mathcal{O}_{X'}(m)$ is already defined over $X$ and thus $\mathcal{O}_{X'}(m)\cong \, ^\omega\mathcal{O}_{X'}(m)^\vee$
iff $m=0$.
\end{proof}

This leads to the following result:

\begin{theorem} The Azumaya algebra
 $$\B=\E nd_{X'} (\mathcal{O}_{X'}\oplus\mathcal{O}_{X'}(m)\oplus\mathcal{O}_{X'}(-m))$$
 of rank 9 over $X'$ permits an involution $*_\B$, which is adjoint to the hermitian form $h$ defined on
 $\mathcal{O}_{X'}\oplus\mathcal{O}_{X'}(m)\oplus\mathcal{O}_{X'}(-m)$.
 The Tits process $\mathcal{J}(\B, \h (\B,*_\B),\Pe,N,*)$ has the following underlying $\sts$-module structure:
$$ \h (\B,*_\B)\cong \mathcal{O}_{X'}^3\oplus \mathcal{O}_{X'}(m)^2\oplus \mathcal{O}_{X'}(-m)^2\oplus
\mathcal{O}_{X'}(2m)\oplus \mathcal{O}_{X'}(-2m)$$
as $\sts$-module and
$$\begin{array}{l}
\Pe\cong \mathcal{O}_{X}(n_1)\oplus \sts(n_2)\oplus \sts(-(n_1+n_2))\\
\oplus \sts(-m+n_1)\oplus \sts(-m+n_2)\oplus \sts(-m-(n_1+n_2))\\
\oplus \sts(m+n_1)\oplus \sts(m+n_2)\oplus \sts(m-(n_1+n_2))\\
\oplus
\sts(-n_1)\oplus \sts(-n_2)\oplus \sts(n_1+n_2)\\
\oplus \sts(m-n_1)\oplus \sts(m-n_2)\oplus \sts(m+n_1+n_2)\\
\oplus \sts(-m-n_1)\oplus \sts(-m-n_2)\oplus \sts(-m+n_1+n_2).\\
\end{array}$$
as $\sts$-module.
\end{theorem}

\begin{proof} We have $\h(\mathcal{O}_{X'},*_\B)=\sts$. $\B$, $\mathcal{O}_X'=\mathcal{O}_{X'}$, $*_\B$
 are suitable for the Tits process $\mathcal{J}(\B, \h (\B,*_\B),\Pe,N,*)$.
 The involution $*_\B$ is given by $\sigma_h$ and

\[ \h (\B,*_\B)\otimes \mathcal{O}_{X'} \cong
\left [\begin {array}{ccc}
\mathcal{O}_{X'} & \mathcal{O}_{X'}(-m) & \mathcal{O}_{X'}(m)\\
\mathcal{O}_{X'}(-m) &  \mathcal{O}_{X'} & \mathcal{O}_{X'}(2m)  \\
\mathcal{O}_{X'}(-m) & \mathcal{O}_{X'}(-2m) &  \mathcal{O}_{X'} \\
\end {array}\right ]
\]
implies the $\sts$-module structure of $\h (\B,*_\B)$. Since
$\mathcal{J}(\B, \h (\B,*_\B),\Pe,N,*)\otimes  \mathcal{O}_{X'} $ contains $\B^+$ it is a first Tits construction
$\mathcal{J}(\B,\Pe_0,N_0)$ where
$$\Pe_0\cong\h om_{\mathcal{O}_{X'}}(\G,\mathcal{O}_{X'}\oplus \mathcal{O}_{X'}(m)\oplus \mathcal{O}_{X'}(-m))$$
for a vector bundle $\G$ of rank 3 with ${\rm det}\,\G\cong \mathcal{O}_{X'}$; i.e.,
$$\G\cong \mathcal{O}_{X'}(n_1)\oplus \mathcal{O}_{X'}(n_2)\oplus \mathcal{O}_{X'}(-(n_1+n_2))$$
($n_i\in\mathbb{Z}$). By [Pu2, Proposition 3], we obtain the assertion.
\end{proof}

\begin{remark} If $\B$ is an Azumaya algebra over $X'$ of rank 9 then either $\B\cong B_0\otimes\mathcal{O}_{X'}$ for some
central simple division algebra over $k$ or
$$\B\cong\E nd _{X'} (\mathcal{O}_{X'}(m_1)\oplus\mathcal{O}_{X'}(m_2)\oplus\mathcal{O}_{X'}(m_3))$$
 with $m_i\in\mathbb{Z}$ arbitrary [Kn1, VII.(3.1.1)].
 The case that $\B$ is defined over $k$ and admits an involution $*_\B$ is covered in Theorem 5.
 The case that $\B$ is not defined over $k$ and $\mathcal{O}_{X'}(m_1)\oplus\mathcal{O}_{X'}(m_2)\oplus\mathcal{O}_{X'}
 (m_3)$ admits a regular hermitian form $*_\B$ is covered in Theorem 6. Suppose
 $$\B\cong\E nd _{X'} (\mathcal{O}_{X'}(m_1)\oplus\mathcal{O}_{X'}(m_2)\oplus\mathcal{O}_{X'}(m_3))$$
 admits an involution of the desired type which is not the adjoint involution of some $h$. Then
 \begin{align*}\h (\B,*_\B)\otimes \mathcal{O}_{X'}\cong\B^+ \cong
  \left [\begin {array}{ccc}
\mathcal{O}_{X'} & \mathcal{O}_{X'}(a) & \mathcal{O}_{X'}(b)\\
\mathcal{O}_{X'}(-a) &  \mathcal{O}_{X'}& \mathcal{O}_{X'}(b-a)  \\
\mathcal{O}_{X'}(-b) & \mathcal{O}_{X'}(a-b) &  \mathcal{O}_{X'} \\
\end {array}\right ]
\end{align*}
($a=m_1-m_2$, $b=m_1-m_3$) implies that
 \[\h (\B,*_\B)\cong \mathcal{O}_{X'}^3\oplus  \mathcal{O}_{X'}(a)\oplus  \mathcal{O}_{X'}(-a)
 \oplus  \mathcal{O}_{X'}(b)\oplus  \mathcal{O}_{X'}(-b)\oplus  \mathcal{O}_{X'}(a-b)\oplus  \mathcal{O}_{X'}(b-a).\]

\end{remark}

\subsection{} Now let $X$ be a nonrational curve over a field $k$.
 Let $k'$ be a separable quadratic field extension of $k$ with ${\rm Gal}(k'/k)=\langle\omega\rangle$
 which is  a splitting field of the quaternion division algebra associated to $X$. Let $X'=X\times_kk'$.
Then $$\B=\E nd_{X'} (\mathcal{O}_{X'}\oplus\mathcal{O}_{X'}(m)\oplus\mathcal{O}_{X'}(-m))$$
 is an Azumaya algebra of rank 9 over $X'$ permitting an involution $*_\B$ (adjoint to the hermitian form defined on
 $\mathcal{O}_{X'}\oplus\mathcal{O}_{X'}(m)\oplus\mathcal{O}_{X'}(-m)$) such that $\h(\mathcal{O}_{X'},*_\B)=\sts$
and $\B$, $\mathcal{O}_X'=\mathcal{O}_{X'}$ and $*_\B$ are suitable for the Tits process $\mathcal{J}(\B, \h (\B,*_\B),\Pe,N,*)$.

\begin{theorem}
 Let $m=2n$ be even. Then
$$ \h (\B,*_\B)\cong \mathcal{O}_{X}^3\oplus \lb(nP_0)^2\oplus \lb(-nP_0)^2\oplus \lb(2nP_0)\oplus \lb(-2nP_0).$$
For a Tits process $\mathcal{J}(\B, \h (\B,*_\B),\Pe,N,*)$, the $\sts$-module structure of $\Pe$ can be computed for
the following cases:\\
(1) If $n_1=2r_1$ and $n_2=2r_2$ are even then
$$\begin{array}{l}
\Pe\cong
 \lb(r_1P_0)\oplus  \lb(r_2P_0)\oplus  \lb(-(r_1+r_2)P_0)\\
\oplus  \lb((-n+r_1)P_0)\oplus  \lb((-n+r_2)P_0)\oplus  \lb((-n-(r_1+r_2))P_0)\\
\oplus  \lb((n+r_1)P_0)\oplus  \lb((n+r_2)P_0)\oplus  \lb((n-(r_1+r_2))P_0)\\
\oplus
 \lb(-r_1P_0)\oplus  \lb(-r_2P_0)\oplus  \lb((r_1+r_2)P_0)\\
\oplus  \lb((n-r_1)P_0)\oplus  \lb((n-r_2)P_0)\oplus  \lb((n+(r_1+r_2))P_0)\\
\oplus  \lb((-n-r_1)P_0)\oplus  \lb((-n-r_2)P_0)\oplus  \lb((n+(r_1-r_2))P_0).
\end{array}$$

\noindent (2) If  $n_1=2r_1+1$ odd and $n_2=2r_2$ even then
$$\begin{array}{l}
\Pe\cong
 tr_{k'/k}(\mathcal{O}_{X'}(1))\otimes  \lb(r_1P_0)\oplus  \lb(r_2P_0)\oplus   tr_{k'/k}(\mathcal{O}_{X'}(-1))\otimes
  \lb(-(r_1+r_2)P_0)\\
\oplus  tr_{k'/k}(\mathcal{O}_{X'}(1))\otimes  \lb((-n+r_1)P_0)\oplus  \lb((-n+r_2)P_0)\oplus \\
 tr_{k'/k}(\mathcal{O}_{X'}(-1))
\otimes  \lb((-n-(r_1+r_2))P_0)\\
\oplus   tr_{k'/k}(\mathcal{O}_{X'}(1))\otimes \lb((n+r_1)P_0)\oplus  \lb((n+r_2)P_0)\oplus \\
 tr_{k'/k}(\mathcal{O}_{X'}(-1))
\otimes  \lb((n-(r_1+r_2))P_0)\\
\oplus \lb(-r_2P_0)\oplus  \lb((r_2-n)P_0)\oplus  \lb((-n-r_2)P_0).
\end{array}$$
(3) If $n_1=2r_1+1$ and $n_2=2r_2+1$ both are odd then
$$\begin{array}{l}
\Pe\cong
 tr_{k'/k}(\mathcal{O}_{X'}(1))\otimes  \lb(r_1P_0)\oplus  tr_{k'/k}(\mathcal{O}_{X'}(1))\otimes \lb(r_2P_0)\oplus
  \lb(-(r_1+r_2+1)P_0)\\
\oplus  tr_{k'/k}(\mathcal{O}_{X'}(1))\otimes  \lb((-n+r_1)P_0)\oplus  tr_{k'/k}(\mathcal{O}_{X'}(1))\otimes
 \lb((-n+r_2)P_0)\oplus \\  \lb((-n-(r_1+r_2+1))P_0)
\oplus \\  tr_{k'/k}(\mathcal{O}_{X'}(1))\otimes \lb((n+r_1)P_0)\oplus
 tr_{k'/k}(\mathcal{O}_{X'}(1))\otimes
 \lb((n+r_2)P_0)\oplus \lb((n-(r_1+r_2+1))P_0)\\
\oplus \lb((r_1+r_2+1)P_0)\oplus \lb((n+r_1+r_2+1)P_0)\oplus \\ \lb((-nr_1+r_2+1)P_0).
\end{array}$$
\end{theorem}

\begin{theorem}  Let $m=2n+1$ be odd.  Then
$$\begin{array}{l}
 \h (\B,*_\B)\cong  \mathcal{O}_{X}^3\oplus \lb(nP_0)\otimes tr_{k'/k}(\mathcal{O}_{X'}(1))\oplus
\lb(-nP_0)\otimes tr_{k'/k}(\mathcal{O}_{X'}(-1))\oplus\\
 \lb((2n+1)P_0)\oplus \lb(-(2n+1)P_0)
 \end{array}$$
is the underlying $\sts$-module structure. For a Tits process $\mathcal{J}(\B, \h (\B,*_\B),\Pe,N,*)$, the $\sts$-module
structure of $\Pe$ can be computed for the following cases:\\
 (1) If $n_1=2r_1$ and $n_2=2r_2$ are even then
\begin{align*}
 \Pe\cong
 \lb(r_1P_0)\oplus  \lb(r_2P_0)\oplus  \lb(-(r_1+r_2)P_0) & \\
\oplus  \lb(-r_1P_0)\oplus  \lb(-r_2P_0)\oplus  \lb((r_1+r_2)P_0) &\\
\oplus  tr_{k'/k}(\mathcal{O}_{X'}(-1))\otimes \lb((-n+r_1)P_0)\oplus &\\ tr_{k'/k}(\mathcal{O}_{X'}(-1))\otimes \lb((-n+r_2)P_0)
\oplus  tr_{k'/k}(\mathcal{O}_{X'}(-1))\otimes \lb((-n-r_1-r_2)P_0) &\\
 \oplus  tr_{k'/k}(\mathcal{O}_{X'}(1))\otimes \lb((n+r_1)P_0) \oplus &\\  tr_{k'/k}(\mathcal{O}_{X'}(1))
 \otimes \lb((n+r_2)P_0)
  \oplus  tr_{k'/k}(\mathcal{O}_{X'}(1))\otimes \lb((n-r_1-r_2)P_0)
\end{align*}
\noindent (2) If w.l.o.g. $n_1=2r_1+1$ is odd and $n_2=2r_2$ even then
$$\begin{array}{l}
\Pe\cong
 tr_{k'/k}(\mathcal{O}_{X'}(1))\otimes \lb(r_1P_0)\oplus  \lb(r_2P_0)\oplus  tr_{k'/k}(\mathcal{O}_{X'}(-1))\otimes
 \lb((r_2-r_1)P_0)\\
\oplus  \lb((-n+r_1)P_0)\oplus  tr_{k'/k}(\mathcal{O}_{X'}(-1))\otimes \lb((-n+r_2)P_0)\oplus  \lb((-(n+r_1+r_2+1))P_0)\\
\oplus  \lb((n+r_1+1)P_0)\oplus  tr_{k'/k}(\mathcal{O}_{X'}(1))\otimes \lb((n+r_2)P_0)\oplus \\
 tr_{k'/k}(\mathcal{O}_{X'}(-1)) \otimes \lb((n-(r_1+r_2))P_0).
\end{array}$$
\noindent (3) If  $n_1=2r_1+1$ and $n_2=2r_2+1$ are all odd then
$$\begin{array}{l}
\Pe\cong
 tr_{k'/k}(\mathcal{O}_{X'}(1))\otimes \lb(r_1P_0) \oplus   tr_{k'/k}(\mathcal{O}_{X'}(1))\otimes \lb(r_2P_0)\oplus
  \lb(-(r_1+r_2+1)P_0)\\
\oplus  \lb((-n+r_1)P_0)\oplus  \lb((-n+r_2)P_0)\oplus  tr_{k'/k}(\mathcal{O}_{X'}(-1))\otimes \lb((-n-(r_1+r_2+1))P_0)\\
\oplus  \lb((n+r_1+1)P_0)\oplus  \lb((n+r_2+1)P_0)\oplus  tr_{k'/k}(\mathcal{O}_{X'}(1))\otimes  \lb((n-(r_1+r_2+1))P_0)\\
\oplus
  \lb(-r_2P_0)\oplus  \lb((n-r_1)P_0)\oplus  \lb((n-r_2)P_0)\oplus
  \lb((-n-r_1-1)P_0)\oplus  \lb((-n-r_2-1)P_0).
\end{array}$$
\end{theorem}

We prove Theorems 7 and 8 together:

\begin{proof}   Since
$\mathcal{J}(\B, \h (\B,*_\B),\Pe,N,*)\otimes  \mathcal{O}_{X'} $ contains $\B^+$, it is a first Tits construction
$\mathcal{J}(\B,\Pe_0,N_0)$, where
$$\Pe_0\cong\h om_{\mathcal{O}_{X'}}(\G,\mathcal{O}_{X'}\oplus \mathcal{O}_{X'}(m)\oplus \mathcal{O}_{X'}(-m))$$
for a vector bundle $\G$ of rank 3 with ${\rm det}\,\G\cong \mathcal{O}_{X'}$; i.e.,
$$\G\cong \mathcal{O}_{X'}(n_1)\oplus \mathcal{O}_{X'}(n_2)\oplus \mathcal{O}_{X'}(-(n_1+n_2))$$
($n_i\in\mathbb{Z}$).
The proof of Theorem 6 gives the $\mathcal{O}_{X'}$-structure of $\B$ and $\Pe\otimes\mathcal{O}_{X'}$
which imply the assertion using [AEJ1] since we have
$$\begin{array}{l}
\Pe\otimes \mathcal{O}_{X'}\cong\\
 \mathcal{O}_{X'}(n_1)\oplus  \mathcal{O}_{X'}(n_2)\oplus  \mathcal{O}_{X'}(-(n_1+n_2))\\
\oplus  \mathcal{O}_{X'}(-m+n_1)\oplus  \mathcal{O}_{X'}(-m+n_2)\oplus  \mathcal{O}_{X'}(-m-(n_1+n_2))\\
\oplus  \mathcal{O}_{X'}(m+n_1)\oplus  \mathcal{O}_{X'}(m+n_2)\oplus  \mathcal{O}_{X'}(m-(n_1+n_2))\\
\oplus
 \mathcal{O}_{X'}(-n_1)\oplus  \mathcal{O}_{X'}(-n_2)\oplus  \mathcal{O}_{X'}(n_1+n_2)\\
\oplus  \mathcal{O}_{X'}(m-n_1)\oplus  \mathcal{O}_{X'}(m-n_2)\oplus  \mathcal{O}_{X'}(m+n_1+n_2)\\
\oplus  \mathcal{O}_{X'}(-m-n_1)\oplus  \mathcal{O}_{X'}(-m-n_2)\oplus  \mathcal{O}_{X'}(-m+n_1+n_2).\\
\end{array}$$
By [Pu2, Proposition 3], $^\omega \mathcal{O}_{X'}(m)\cong \mathcal{O}_{X'}(-m)$ if $m$ is odd.
  A case-by-case study yields the assertion.
\end{proof}

It is not clear if all these cases occur.

\begin{theorem} Assume that, in the situation of  Theorem 7,
 $k'$ is not  a splitting field of the quaternion division algebra associated to $X$. Then
$$\B=\E nd_{X'} (\mathcal{O}_{X'}\oplus\lb(mP_0)\oplus\lb(-mP_0))$$
 is an Azumaya algebra of rank 9 over $X'$ permitting an involution $*_\B$ adjoint to the hermitian form defined on
 $\mathcal{O}_{X'}\oplus\lb(mP_0)\oplus\lb(-mP_0)$.  The Tits process $\mathcal{J}(\B, \h (\B,*_\B),\Pe,N,*)$
 has the following underlying $\sts$-module structure:
$$ \h (\B,*_\B)\cong \mathcal{O}_{X}^3\oplus \lb(mP_0)^2\oplus \lb(-mP_0)^2\oplus \lb(2mP_0)\oplus \lb(-2mP_0)$$
as $\sts$-module and either
$$\begin{array}{l}
\Pe\cong \lb(n_1P_0)\oplus\lb(n_2P_0)\oplus\lb(-(n_1+n_2)P_0)\\
\oplus\lb((-m+n_1)P_0)\oplus\lb((-m+n_2)P_0)\oplus\lb(-m-(n_1+n_2)P_0)\\
\oplus\lb((m+n_1)P_0)\oplus\lb((m+n_2)P_0)\oplus\lb((m-(n_1+n_2))P_0)\\
\oplus
\lb(-n_1P_0)\oplus\lb(-n_2P_0)\oplus\lb((n_1+n_2)P_0)\\
\oplus\lb((m-n_1)P_0)\oplus\lb((m-n_2)P_0)\oplus\lb((m+n_1+n_2)P_0)\\
\oplus\lb((-m-n_1)P_0)\oplus\lb((-m-n_2)P_0)\oplus\lb((-m+n_1+n_2)P_0).
\end{array}$$
for arbitrarily chosen $n_i\in\mathbb{Z}$, or
$$\begin{array}{l}
\Pe\cong \lb((2n+1)P_0)\oplus\lb(-nP_0)\otimes tr_{l/k}(\mathcal{O}(-1))\oplus\\
\lb((2n+m+1)P_0)\oplus \lb((m-n)P_0)\otimes tr_{l/k}(\mathcal{O}(-1))\oplus\\
\lb((2n-m+1)P_0)\oplus \lb((-m-n)P_0)\otimes tr_{l/k}(\mathcal{O}(-1)) \oplus\\
 \lb(-(2n+1)P_0)\oplus\lb(nP_0)\otimes tr_{l'/k'}(\mathcal{O}(1))\oplus\\
\lb(-(2n+m+1)P_0)\oplus \lb(-(m-n)P_0)\otimes tr_{l/k}(\mathcal{O}(1))\oplus\\
\lb(-(2n-m+1)P_0)\oplus \lb(-(-m-n)P_0)\otimes tr_{l/k}(\mathcal{O}(1))
\end{array}$$
for arbitrarily chosen $n\in\mathbb{Z}$, where $l$ is a separable quadratic splitting field of $D_0$.
\end{theorem}

\begin{proof}  We have $\h(\mathcal{O}_{X'},*_\B)=\sts$
and $\B$, $\mathcal{O}_X'=\mathcal{O}_{X'}$ and $*_\B$
 are suitable for the Tits process $\mathcal{J}(\B, \h (\B,*_\B),\Pe,N,*)$. Since
$\mathcal{J}(\B, \h (\B,*_\B),\Pe,N,*)\otimes  \mathcal{O}_{X'} $ contains $\B^+$ it is a first Tits construction
$\mathcal{J}(\B,\Pe_0,N_0)$ where
$$\Pe_0\cong\h om_{\mathcal{O}_{X'}}(\G,\mathcal{O}_{X'}\oplus \lb(m)\oplus \lb(-m))$$
for a vector bundle $\G$ of rank 3 with ${\rm det}\,\G\cong \mathcal{O}_{X'}$; i.e.,
$$\G\cong \lb(n_1P_0)\oplus \lb(n_2P_0)\oplus \lb(-(n_1+n_2)P_0)$$
or
$$\G\cong \lb(-(2n+1)P_0)\oplus \lb(nP_0)\otimes tr_{l'/k'}(\mathcal{O}_{X_{l'}}(1))$$
for a quadratic field extension $l'/k'$ which splits the quaternion division algebra associated with $X'$
($n_i\in\mathbb{Z}$).
In the first case, the $\mathcal{O}_{X'}$-structure of $\B$ and $\Pe\otimes\mathcal{O}_{X'}$ can be computed as in the proof
of Theorem 8 and in the second, we get
$$\begin{array}{l}
\Pe_0\cong \lb((2n+1)P_0)\oplus\lb(-nP_0)\otimes tr_{l'/k'}(\mathcal{O}(-1))\oplus\\
\lb((2n+m+1)P_0)\oplus \lb((m-n)P_0)\otimes tr_{l'/k'}(\mathcal{O}(-1))\oplus\\
\lb((2n-m+1)P_0)\oplus \lb((-m-n)P_0)\otimes tr_{l'/k'}(\mathcal{O}(-1))
\end{array}$$
and $\Pe\otimes\mathcal{O}_{X'}\cong\Pe_0\oplus\Pe_0^\vee$.
These imply the assertion using [AEJ1] and the fact that every line bundle over $X'$ is already defined over $X$ here.
\end{proof}

The cases treated in the above theorem are the only ones where the hermitian form $h:\E\to^\omega\E^\vee$
on $\E$ is not defined over $k$ and  $\B=\E nd_{\mathcal{O}_{X'}}(\E)$, if $k$
has characteristic 0.
It is not clear, however, whether all cases appear, since we only worked with necessary conditions when restricting the
module structure of $\Pe$.

\section{Elliptic curves}

The advantage of working over elliptic curves instead of curves of genus zero is that there are
also bundles of degree higher than 2 which are indecomposable which contribute to
 more interesting examples of Jordan algebras. The vector bundles are well-understood, we will use
the results and terminology from Atiyah [At] and Arason, Elman and Jacob [AEJ1,2,3].

Let $k$ be a field of characterisitc not 2 or 3. An elliptic curve $X/k$ can be described by a
Weierstra{\ss} equation of the form
$$y^2 = x^3 + b_2 x^2 + b_1 x + b_0 \qquad (b_i \in k)$$
with the infinite point as base point $O$. Let $q(x) = x^3 + b_2 x^2+ b_1 x + b_0$ be the defining polynomial in $k[x]$. The $k$-rational
points of order 2 on $X$ are the points $(a, 0)$, where $a \in k$ is a root of $q(x)$.
 Let $K = k(X) = k(x, \sqrt{q(x)})$ be the function field of $X$. We
distinguish three different cases (cf. [AEJ3]).

\smallskip
{\it Case I.} $X$ has three $k$-rational points of order 2 which is
equivalent to $_2 \Pic (X) \cong {\Bbb Z}_2 \times {\Bbb Z}_2$. Write $q(x) =
(x-a_1) (x-a_2) (x-a_3)$ and $_2 \Pic (X) = \{ \mathcal{O}_X, \mathcal{L}_1, \mathcal{L}_2, \mathcal{L}_3 \}$
where $\mathcal{L}_i$ corresponds with the point $(a_i, 0)$ for $i = 1,2,3$.

\smallskip
{\it Case II.} $X$ has one $k$-rational point of order 2 which is equivalent to
$_2 \Pic (X) \cong {\Bbb Z}_2$. Write $q(x) = (x-a_1) q_1 (x)$ and $_2 \Pic (X) =
\{ \mathcal{O}_X, \mathcal{L}_1 \}$ with $\mathcal{L}_1$ corresponding with $(a_1, 0)$.
Define $\ell_2  = k (a_2)$ with $a_2$ a root of $q_1, K_2  = K \otimes_k \ell_2$.

\smallskip\
{\it Case III.} $X$ has no $k$-rational point of order 2 which is equivalent
to $_2 \Pic (X) = \{ \mathcal{O}_X \}$. Define $\ell_1  = k (a_1)$ with
$a_1$ a root of the irreducible polynomial $q(x)$, $K_1 = K \otimes_k
\ell_1$ and let $\Delta (q) = (a_1 -a_2)^2 (a_1 - a_3)^2 (a_2 - a_3)^2$ be the discriminant of $q$.

\smallskip
Correspondingly, $X/k$ is called {\it of type I, II} or {\it III}.
We denote the absolutely indecomposable vector bundle of rank $r$ and degree 0 with nontrivial global
sections (which is uniquely determined up to isomorphism [At, Theorem 5]) by $\F_r$. Let $\overline{k}$ be an algebraic
closure of $k$ and let $\overline{X}=X\times_k \overline{k}$.

For simplicity, we assume from now on that $k$ has characteristic zero.
 Then let $\N_i$ denote a line bundle of order 3 on $X$.
There is a nondegenerate cubic form on $\N_i$, which is uniquely determined up to an invertible factor in $k$
[Pu2, Lemma 1]. Now
$_3 \Pic (\overline{X}) = \{\N_i\,|\, 0\leq 1\leq 8\}$ where $\N_0= \mathcal{O}_{\overline{X}} $ [At, Lemma 22].
Hence $_3 \Pic (X) = \{\N_i\,|\, 0\leq 1\leq m\}$ for some even integer $m$, $ 0\leq m\leq 8$, where $\N_0= \mathcal{O}_X $.

Let $l/k$ be a finite field extension. For a vector bundle $\mathcal{N}$ on
$Y  = X \times_k \ell$, the direct image $f_* \mathcal{N}$ of $\mathcal{N}$ under the
projection morphism $f \colon Y \to X$ is a vector bundle on $X$ denoted by
$tr_{l/k} (\mathcal{N})$.

\begin{lemma} Let $l/k$ be a quadratic field extension with ${\rm Gal}(l/k)=\langle \omega\rangle$.
 Let $X_l=X\times_k l$.\\
 (i) If $X$ has type I or III, or type II and $l\not\cong l_2$, then
  every  Tits process over $X$ starting with $\B=\mathcal{O}_{X_l}$ and $*_\B=\omega$ is defined over $k$.\\
  (ii) Let $X$ be of type II and $l\cong l_2$. If there is a line bundle $\N_i$ over $X_l$ of order 3
  which is not defined over $X$ and satisfies $^\omega\N_i\cong\N_i^\vee$, then there is a Tits process
$\J=\mathcal{J}(\mathcal{O}_{X_l},\sts,\N_i,N,*)\cong\sts\oplus\N_i$ which is not defined over $X$. Otherwise every Tits process
 starting with $\B=\mathcal{O}_{X_l}$ and $*_\B=\omega$ is defined over $X$.
\end{lemma}

\begin{proof} (i) If $X$ is of type III, then so is $X_l$ and thus each $\Pe\in {\rm Pic}\,\mathcal{O}_{X_l}$ is defined
over $X$ which implies $^\omega\Pe\cong\Pe$.
 If $\Pe\in{\rm Pic}\,\mathcal{O}_{X_l}$ is $\sts$-admissible, then $^\omega\Pe\cong\Pe^\vee$  and hence $\Pe\cong\Pe^\vee$. Therefore $\Pe\cong\mathcal{O}_{X_l}$ and so the Tits process
 $\J=\mathcal{J}(\mathcal{O}_{X_l},\sts,\Pe,N,*)$ is globally free as $\sts$-module, implying that it is defined over $k$
  [Pu2, Lemma 1].

 If $X$ is of type I then so is $X_l$ and thus each $\Pe\in {\rm Pic}\,\mathcal{O}_{X_l}$ is defined over $X$.
If $X$ is of type II and $l\not\cong l_2$, then so is $X_l$ and thus each $\Pe\in {\rm Pic}\,\mathcal{O}_{X_l}$ is defined over $X$.
Again this implies that in both cases, the Tits process
 $\J=\mathcal{J}(\mathcal{O}_{X_l},\sts,\Pe,N,*)$ is defined over $k$.
\\ (ii) Since $l\cong l_2$, $X_l$ is of type I. Now
  $N_{\mathcal{O}_{X_l}}(\lb\otimes \mathcal{O}_{X_l})\cong \lb^3\cong \mathcal{O}_{X_l}$ for a
  line bundle $\lb$ over $X_l$,
 if and only if $\lb\cong\N_i$ for some $i$. Let $\Pe\in {\rm Pic}\,\mathcal{O}_{X_l}$ be
 $\sts$-admissible, then $\Pe\cong \N_i$ and $^\omega \N_i\cong\N_i^\vee$.
If $\N_i$ is already defined over $X$, this implies $\N_i\cong \mathcal{O}_{X_l}$. (If $\N_i$ is nontrivial and defined over $X$,
this would yield the contradiction that $ \N_i\cong\N_i^\vee$.)
\end{proof}

\begin{lemma}  Every first Tits construction starting
with $\mathcal{O}_X$ is isomorphic to $\mathcal{J}(\sts,\N_i, N_i)$ with $N_i:\N_i\to \sts$ a nondegenerate
cubic form on $\N_i$.
\end{lemma}

By the Theorem of Krull-Schmidt, $\mathcal{J}(\sts,\N_i, N_i)\not\cong\mathcal{J}(\sts,\N_j, N_j)$ if $\N_i\not
\cong\N_j$ and $\N_i\not\cong\N_j^\vee$.
Hence if $m=2$ there are at least two non-isomorphic such constructions, if $m=4$ at least three,
if $m=6$ at least four and if $m=8$ there are at least five non-isomorphic ones.

We thus have found examples of commutative associative algebras $\A=\J (\sts, \lb,  N)$ over $X$, which are not defined over $k$,
where $\A(P)$ is a cubic  \'{e}tale algebra over $k(P)$ for all $P\in {\rm Spec}\,X$.

\begin{remark} (analogously to [AEJ3, p.~11]) To a rational point $P\in X$ there corresponds an absolutely indecomposable
vector bundle $\E_{(P)}$ of rank 3 and degree 1 such that ${\rm det}\,\E_{(P)}\cong \M_{(P)}$,
 where $\M_{(P)}\in\Omega(1,1)$
 is the line bundle corresponding to $P\in X$ (cf. [At], [Ti]). There is an exact sequence
$$0\to\sts\to \E_{(P)}\to \M_{(P)}\to 0$$
of $\sts$-modules. Let $\A_{(P)}=\E nd(\E_{(P)})$, then the norm of $\A_{(P)}$ is a nondegenerate cubic form
on $\A_{(P)}$. $\overline{\A_{(P)}}$ is the direct sum of 9 line bundles over $\overline{X}$ of order dividing
 3. $\sts$ embeds into $\A_{(P)}$ canonically. Let $\N$ be a line bundle of order 3, then
 $\N\otimes \E_{(P)}$ is an absolutely indecomposable vector bundle of rank 3 and degree 1 with
 ${\rm det}(\N\otimes \E_{(P)})\cong {\rm det}\,\E_{(P)}$. Hence $\N\otimes \E_{(P)}\cong \E_{(P)}$.
 Tensoring with $\E_{(P)}^\vee$ and identifying $\E_{(P)}\otimes \E_{(P)}^\vee$ and $\A_{(P)}$ we get an
 isomorphism $\N\otimes \E_{(P)}\cong \E_{(P)}$ and in particular, an embedding of $\N=\N\otimes_X\sts\subset
 \N\otimes\A_{(P)}$ into $\A_{(P)}$. This induces a cubic form on $\N$.

 It would be nice to be able to compute this cubic form along similar lines as it was done in
 [AEJ3, 3.1] for a quadratic form.
\end{remark}

\section{ \'Etale First Tits constructions}

We look at some cases where we can compute the \'etale First Tits construction.

\begin{lemma} Let $\A=\sts\times\sts\times\sts$. The \'etale first Tits construction
$\mathcal{J}(\A,\Pe,N)$ is a Jordan algebra over $X$ of rank 9 with underlying module structure
$$\sts^3\oplus \lb\oplus \mathcal{M}\oplus (\lb^\vee\otimes\mathcal{M}^\vee)\oplus
\lb^\vee\oplus \mathcal{M}^\vee\oplus (\lb\otimes\mathcal{M}),$$
for some line bundles $\lb, \mathcal{M}$ over $X$.
\end{lemma}

\begin{proof} Every left $\A$-module $\Pe$ of rank one with $N_\A(\Pe)\cong\sts$ satisfies
$$\Pe\cong \lb\oplus \mathcal{M}\oplus (\lb^\vee\otimes\mathcal{M}^\vee)$$
for some line bundles $\lb, \mathcal{M}$ over $X$ [Pu2, Example 8].
Choose an isomorphism $\alpha:\lb\otimes \mathcal{M}\otimes (\lb^\vee\otimes\mathcal{M}^\vee)\to\sts$, then
$N(x,y,z)=\alpha(x\otimes y\otimes z)$
defines a norm on $\Pe$ and $\mathcal{J}(\A,\Pe,N)$ has the underlying module structure claimed in the assertion.
\end{proof}

\begin{proposition}  Let $k'/k$ be a  quadratic field extension of $k$ with ${\rm Gal}(k'/k)=\langle \sigma\rangle$.
Let $\torus=k'\otimes \sts$.
For $\A=\sts\times \torus$, every left $\A$-module $\Pe$ of rank one with $N_\A(\Pe)\cong\sts$ is
 isomorphic to
$$\lb^{\vee}\otimes\lb^\vee\oplus \torus\otimes \lb$$
as an $\sts$-module, for any line bundle $\lb$ over $X$. In particular, the  \'etale first Tits construction
$\J(\A,\Pe,N)$ has the direct sum of line bundles
$$ \sts^3\oplus \lb^{\vee}\otimes\lb^\vee\oplus \lb\oplus \lb\oplus \lb\otimes\lb\oplus \lb^\vee\oplus \lb^\vee$$
 as underlying $\sts$-module.
\end{proposition}

\begin{proof} By [Pu2, Lemma 5], every left $\A$-module $\Pe$
of rank one satisfies $\Pe\cong \lb\oplus\Pe_0$ with $\Pe_0\in {\rm Pic}\torus$ and a line bundle $\lb\in {\rm Pic}\,X$, and
 $ N_\A(\Pe)\cong \sts$ iff $\lb^\vee\cong N_\torus(\Pe_0)$, where $N_\torus$ denotes the norm on $ \torus$.
   Let $X'=X\times_kk'$. Identify ${\rm Pic}\,X'$ and ${\rm Pic}\,\torus$.
 First, consider the case that $X$ has type I or III, or that $X$ has type II and $k'$ is not isomorphic to $l_2$.
Then the canonical map $${\rm Pic}\,X\to {\rm Pic}\,X', \,\,\lb\to\lb\otimes \mathcal{O}_{X'}$$
 is bijective. Hence every $\Pe_0\in {\rm Pic}\,\torus$ is of the form $\Pe_0
\cong (k'\otimes \sts)\otimes \lb$. We get
$$N_\torus(\Pe_0)\cong N_\torus(\torus\otimes \lb)\cong\lb\otimes\lb.$$
 Hence each $\Pe$, which can be used for the first Tits construction, is isomorphic to
 $\Pe\cong \lb^{\vee}\otimes\lb^\vee\oplus \torus\otimes \lb$.

Now let $X$ have type II and $k'\cong l_2$. Then $X'$ is of type II and, by the above,
$$\Pe\otimes\mathcal{O}_{X'}\cong \lb'^{\vee}\otimes\lb'^\vee\oplus\lb'\oplus\lb'$$
 for a line bundle $\lb'$ over $X'$. If this line bundle is defined over $X$, then
 $\Pe\cong \lb^{\vee}\otimes\lb^\vee\oplus \torus\otimes \lb$ for some line bundle $\lb.$ If it is not defined
 over $X$, we obtain a contradiction, since we know that we must have $\Pe\cong \lb\oplus\Pe_0$
 as $\sts$-module for some line bundle $\lb$ over $X$.
\end{proof}

\begin{example}
Let $\A=\sts\times\torus$ for some  quadratic \'etale algebra $\torus$ over $X$ which is not defined over $k$.
 Then $X$ has type I or II.
\\ (i) Suppose $X$ has type I and $\torus={\rm Cay}(\sts,\mathcal{L}_i,N_{\lb_i})$ for $i=1,2$ or 3.
 For every line bundle $\M$ over $X$, $\M\otimes\torus \in {\rm Pic}\torus$ and $N_\torus (\M\otimes\torus )
 \cong \M^2$. Thus $\Pe\cong\M^{\vee 2}\oplus \M\otimes\torus$ is a left $\A$-module  of rank one such that
$ N_\A(\Pe)\cong \M^{\vee 2}\otimes N_\torus (\M\otimes\torus )\cong  \sts$ and hence carries a cubic norm $N$.
Therefore
$$\J(\A,\Pe,N)\cong \sts\oplus\sts\oplus \lb_i\oplus \M^{\vee 2}\oplus \M\oplus \M\otimes\lb_i \oplus
\M^{ 2}\oplus \M^\vee\oplus \M^\vee\otimes\lb_i$$
as $\sts$-module.
\\ (ii) Suppose $X$ has type $II$ and $\torus={\rm Cay}(\sts,\mathcal{L}_1,N_{\lb_1})$.
Now
$tr_{l_2/k}(\lb_2) \in {\rm Pic}\torus$ as well [Pu1, 3.3]. For every
line bundle $\M$ over $X$, we have  $\M\otimes tr_{l_2/k}(\lb_2) \in {\rm Pic}\torus$
  and $N_\torus(\M\otimes tr_{l_2/k}(\lb_2))\cong \M^2$.  Thus $\M^{\vee 2}\oplus tr_{
l_2/k}(\M\otimes\lb_2)$ is a left $\A$-module $\Pe$ of rank one such that
$ N_\A(\Pe)\cong \M^{\vee 2}\otimes N_\torus (tr_{ l_2/k}(\M\otimes\lb_2)\cong  \sts$ and hence carries a cubic norm $N$.
Therefore
$$\J(\A,\Pe,N)\cong  \sts\oplus\sts\oplus \lb_3\oplus \M^{\vee 2}\oplus  tr_{l_2/k}(\M\otimes\lb_2) \oplus \M^{ 2}\oplus
 tr_{l_2/k}(\M^\vee\otimes\lb_2)$$
for some line bundle $\M$.
\end{example}

\begin{proposition} Let $k'/k$ be a cubic Galois field extension of $k$ and ${\rm Gal}(k'/k)=\{id,\sigma_1,\sigma_2\}$.
Let $\A=k'\otimes \sts$. \\
(i) Let $X$ be of type I or II, or let $X$ have type III and $k'$ be not isomorphic
to $l_1$. Every left $\A$-module $\Pe$ of rank one with $N_\A(\Pe)\cong\sts$ is isomorphic to
 $\A\otimes \N_i$.\\
 (ii) Let $X$ have type III and $k'\cong l_1$. Every left
 $\A$-module $\Pe$ of rank one with $N_\A(\Pe)\cong\sts$ is a direct sum of line bundles isomorphic to
$\A\otimes \N_i$, or indecomposable and isomorphic to $tr_{l_1/k}(\N_i')$, with $\N_i'$ a line bundle of order 3 over $X'$,
which is not defined over $X$ and satisfies $\N_i'\cong^{\sigma_i} \N_i'$ for $i=1,2$.
\\ In particular, the  \'etale first Tits construction $\J(\A,\Pe,N)$ has the direct sum of line bundles
$$\sts^3\oplus \N_i^3\oplus [\N_i^\vee]^3$$
as underlying module structure, or
$$\sts^3\oplus tr_{l_1/k}(\N_i')\oplus tr_{l_1/k}(\N_i'^\vee), $$
if $X$ has type III and $k'\cong l_1$, with $\N_i'$ as in (ii).
\end{proposition}

\begin{proof} Let $X'=X\times_kk'$. Identify ${\rm Pic}\,X'$ and ${\rm Pic}\,\A$.\\
(i) The map
$${\rm Pic}\,X\to {\rm Pic}\,X', \,\, \lb\to\lb\otimes \mathcal{O}_{X'}.$$
 is bijective. Therefore every $\Pe\in {\rm Pic}_l\,\A$ is of the form $\Pe
 \cong (k'\otimes \sts)\otimes \lb$.
If $N$ is the norm of $k'\otimes_k \sts$, then
$$\sts\cong N(\Pe)\cong N((k'\otimes \sts)\otimes \lb)\cong\lb^3$$
implies that $\lb$ has order $3$. Hence each $\Pe$ which can be used for the first Tits construction
satisfies $\Pe\cong \A\otimes \N_i$.
\\ (ii) By the above,
$$\Pe\otimes\mathcal{O}_{X'}\cong \N_i'\oplus \N_i'\oplus \N_i'$$
 for a line bundle $\N_i'$ over $X'$ of order 3. If this line bundle is defined over $X$, then
 $\Pe\cong \N_i\oplus \N_i\oplus \N_i $ for a line bundle $\N_i$ of order 3 over $X$.
If it is not defined over $X$, then $\Pe\cong tr_{l_1/k}(\N_i')$ is indecomposable. Moreover, then
$ \N_i'\oplus \N_i'\oplus \N_i' \cong  \N_i'\oplus^{\sigma_1} \N_i'\oplus^{\sigma_2} \N_i'$ implies that
$\N_i'\cong^{\sigma_i} \N_i'$ for $i=1,2$, using Krull-Schmidt.
\end{proof}

It is not clear if and when case (ii) can happen. The well-known situation for line bundles of order 2 over $X$ is an indication
that, if those of order 3 behave similarly, the condition that $\N_i'\cong^{\sigma_i} \N_i'$ for $i=1,2$, which is necessary
for the line bundle $\N_i'$ which is not supposed to be defined over $X$, might never be true.

\begin{example} Let $\A=\mathcal{J}(\sts,\N_i,N_i)$ with non-trivial $\N_i$ be a commutative associative
 Jordan algebra over $X$ of rank 3. For every $0\leq j\leq m$,
the first Tits construction $\mathcal{J}(\A,\N_j, N)$ yields a Jordan algebra over $X$ of rank 9
for a suitable cubic form $N:\N_i\otimes\A\to\sts$. Therefore the underlying $\sts$-module structure
of such a construction is a direct sum of line bundles of order 3:
$$(\sts\oplus \N_i\oplus \N_i^2)\oplus (\N_j\oplus \N_i\otimes\N_j\oplus \N_i^2\otimes\N_j)\oplus
(\N_j^2\oplus \N_i^2\otimes\N_j^2\oplus \N_i\otimes\N_j^2).$$
\end{example}

\begin{remark}
Let $\E$ be a commutative associative $\sts$-algebra of constant rank 3 such that $\E(P)$ is a cubic
 \'etale $k(P)$-algebra for all $P\in X$. Let $\xi$ be the generic point of the elliptic curve $X$.
\\ (i) If the two  \'etale first Tits constructions $\J=\mathcal{J}(\E,\Pe,\beta N)$ and
$\J'=\mathcal{J}(\E,\Pe,\beta'N)$ are isomorphic over $X$, then $\J_\xi\cong\J'_\xi$ implies
$(\Pe,N)_\xi\cong (\E_\xi,\alpha N_{\E_\xi})$ and so $\beta \equiv {\beta'}{\rm mod}
\, N_{\E_\xi}(\E_\xi^\times)$ or $\beta \equiv {\beta'}^2{\rm mod} \, N_{\E_\xi}(\E_\xi^\times)$ by [P-T, 4.3].
\\ (ii) If the two  \'etale first Tits constructions $\J=\mathcal{J}(\E,\Pe,N)$ and
$\J'=\mathcal{J}(\E,\Pe',N')$ are isomorphic over $X$, then $\J_\xi\cong\J'_\xi$ implies
$(\Pe,N)_\xi\cong (\E_\xi,\alpha)$ and $ (\Pe',N')_\xi\cong (\E_\xi,\alpha')$ with $\alpha \equiv {\alpha'}{\rm mod}
\, N_{\E_\xi}(\E_\xi^\times)$ or $\alpha \equiv {\alpha'}^2{\rm mod}
\, N_{\E_\xi}(\E_\xi^\times)$.
\end{remark}

\section{Albert algebras over elliptic curves of the kind $\h_3(\comp)$}

Although we do not have a full classifiacation of the octonion algebras over an elliptic curve yet, the information
we have so far serves us to give a class of examples of Albert algebras of the kind $\h_3(\comp)$.
 Obviously, the underlying $\sts$-module structure of the Albert algebra $\h_3(\comp)$ is given by
$$\sts^3\oplus \comp^3. $$
If $\comp$ is a split octonion algebra over $X$, then $\h_3(\comp)$ is a first Tits construction starting with the Azumaya
algebra ${\rm Mat}_3(\sts)$ [Pu2, Proposition 4]. Using [Pu1, Proposition 4.1] we obtain:

\begin{theorem} Let $\comp$ be a split octonion algebra. Then the underlying $\sts$-module structure of
the Albert algebra $\h_3(\comp)$ is given by one of the following:\\
(i) $\sts^{27}$ (and $\h_3(\comp)$ is defined over $k$), if $\comp\cong {\rm Zor}\,k\otimes\sts$;
\\ (ii) $$\sts^9\oplus [\lb\oplus \N\oplus \lb^\vee\otimes \N^\vee]^3\oplus [\lb^\vee\oplus \N^\vee\oplus \lb\otimes \N]^3$$
for line bundles $\lb$ and $\N$ over $X$ which are not both isomorphic to $\sts$, if $\comp\cong {\rm Zor}(\lb\oplus \N\oplus \lb^\vee\otimes \N^\vee)$;
\\ (iii) $$\sts^9\oplus [ \E \oplus {\rm det}\,\E^\vee]^3\oplus [ \E^\vee \oplus {\rm det}\,\E]^3$$
 for an absolutely
indecomposable $\E\in\Omega (2,d)$, if $\comp\cong {\rm Zor}(\E\oplus {\rm det}\,\E^\vee)$.
 Given two absolutely indecomposable $\E,$ $\E'\in \Omega (2,d)$, the corresponding Albert
algebras are not isomorphic, if $\E\not\cong\E'$ and $\E\not\cong\E'^\vee$.
\\ (iv) $$\sts^9\oplus [tr_{l/k}(\N)\oplus ({\rm det}\,tr_{l/k}(\N))^\vee]^3
\oplus [tr_{l/k}(\N)^\vee\oplus ({\rm det}\,tr_{l/k}(\N))]^3,$$
if $\comp\cong {\rm Zor}(tr_{l/k}(\N)\oplus ({\rm det}\,tr_{l/k}(\N))^\vee)$. Here, $l/k$ is a quadratic
field extension and $\N$ a line bundle over $X_l=X\times_kl$ not defined over $X$.
 Given two  indecomposable vector bundles $tr_{l/k}(\N)$ and $tr_{l'/k}(\N')$, the corresponding Albert
algebras are not isomorphic, if $tr_{l/k}(\N)\not\cong tr_{l'/k}(\N')$ and $tr_{l/k}(\N)\not\cong tr_{l'/k}(\N')^\vee$.
\\ (v)  $$\sts^9\oplus [\N_i\otimes\F_3]^3 \oplus [\N_i^\vee\otimes\F_3]^3,$$ if $\comp\cong {\rm Zor}(\N_i\otimes\F_3,\alpha)$.
For two non-isomorphic line bundles $\N_i$ and $\N_j$ of order 3 with $\N_i\not\cong\N_j^\vee$,
the corresponding Albert algebras are not isomorphic.
\\ (vi)  $$\sts^9\oplus [tr_{l/k}(\N)]^3 \oplus [tr_{l/k}(\N)^\vee]^3,$$ if $\comp\cong {\rm Zor}(tr_{l/k}(\N))$.
Here, $l/k$ is a cubic field extension and $\N$ a line bundle over $X_l=X\times_kl$ not defined over $X$.
Given two  indecomposable vector bundles $tr_{l/k}(\N)$ and $tr_{l'/k}(\N')$, the corresponding Albert
algebras are not isomorphic, if $tr_{l/k}(\N)\not\cong tr_{l'/k}(\N')$ and $tr_{l/k}(\N)\not\cong tr_{l'/k}(\N')^\vee$.
\\ The Albert algebras in cases (i) to (vi) are mutually non-isomorphic.
\end{theorem}

Using [Pu1, Proposition 4.1] we obtain:

\begin{theorem} Let $\comp$ be a Cayley-Dickson doubling of the quaternion algebra $\E nd_X(\F_2)$. Then the underlying $\sts$-module structure of
the Albert algebra $\h_3(\comp)$ is given by one of the following:\\
(i) $$\sts^9\oplus [\F_3\oplus \F_3]^3$$
\noindent (ii) $$\sts^6\oplus [\F_3\oplus\lb_i\oplus\lb_i\otimes \F_3]^3$$
with $1\leq i\leq 3$. For two non-isomorphic line bundles $\lb_i,$ $\lb_j$, $i\not=j$, the
corresponding Albert algebras are not isomorphic.
\\ (iii) $$\sts^6\oplus [\F_3\oplus \lb\otimes\F_2\oplus  \lb^\vee\otimes\F_2]^3$$ for any line bundle $\lb$ over $X$;
\\ (iv) $$\sts^6\oplus [\F_3\oplus tr_{l/k}(\N\otimes\F_2)]^3,$$ where $l/k$ is a quadratic field extension
with ${\rm Gal}(l/k)=\langle\sigma\rangle$ and
$\N$ a line bundle over $X_l=X\times_kl$ not defined over $X$, which is not self-dual and satisfies $\N^\vee\cong^\sigma\N$.
\\ The Albert algebras in cases (i) to (iv) are mutually non-isomorphic. The ones in cases
(ii) to (iv) are not isomorphic to those listed in Theorem 9.
\end{theorem}

Let $\comp$ be a Cayley-Dickson doubling of the quaternion algebra $\E nd_X(\E)$ for
an absolutely indecomposable vector bundle $\E\in \Omega(2,1)$.
Using [Pu1, Proposition 4.1] we obtain:

\begin{theorem} Let $X$ have type I. Then the underlying $\sts$-module structure of
the Albert algebra $\h_3(\comp)$ is given by one of the following:\\
(i)   $$\sts^9\oplus \lb_1^6\oplus\lb_2^6\oplus\lb_3^6,$$
\noindent (ii) $$\sts^6\oplus [ \lb_1\oplus\lb_2\oplus\lb_3\oplus \lb\otimes\E\oplus  \lb^\vee\otimes\E^\vee]^3,$$
for any line bundle $\lb$ over $X$,
\\ (iii) $$\sts^6\oplus [ \lb_1\oplus\lb_2\oplus\lb_3\oplus tr_{l/k}(\N\otimes (\E^\vee\otimes\mathcal{O}_{X_l}))]^3.$$
Here, $l/k$ is a quadratic field extension with ${\rm Gal}(l/k)=\langle\sigma\rangle$ and $\N$ a line bundle over
$X_l=X\times_kl$ not defined over $X$, which is not self-dual and satisfies $\N\cong ^\sigma\N\cong ({\rm det}\,\E
\otimes\mathcal{O}_{X_l})$.\\
 The Albert algebras in case (iii) are mutually non-isomorphic to those in case (i) and (ii). Within (i) and (ii) there is some
obvious overlap.
 The algebras are all not isomorphic to those listed in Theorem 10 and also not isomorphic to those listed
in Theorem 9, cases (iii) to (vi).
\end{theorem}

\begin{theorem} Let $X$ have type II. Then the underlying $\sts$-module structure of
the Albert algebra $\h_3(\comp)$ is given by one of the following:\\
 (i) $$\sts^9\oplus [ \lb_1^2\oplus tr_{l_2/k}(\lb_2)^2]^3,$$
\noindent (ii) $$\sts^9\oplus [ \lb_1\oplus tr_{l_2/k}(\lb_2)\oplus \lb\otimes\E\oplus \lb^\vee\otimes\E^\vee]^3,$$
for any line bundle $\lb$ over $X$,
\\ (iii) $$\sts^9\oplus [ \lb_1\oplus tr_{l_2/k}(\lb_2)\oplus tr_{l/k}(\N\otimes( \E^\vee\otimes\mathcal{O}_{X_{l_2}}))]^3.$$
Here, $l/k$ is a quadratic field extension with ${\rm Gal}(l/k)=\langle\sigma\rangle$ and $\N$ a line bundle over
$X_l=X\times_kl$ not defined over $X$, which is not self-dual and satisfies $\N\cong ^\sigma\N\cong ({\rm det}\,\E
\otimes\mathcal{O}_{X_l})$.
\\ The Albert algebras in cases (i) to (iii) are mutually non-isomorphic.
 The algebras are all not isomorphic to those listed in Theorems 10 and 11.
\end{theorem}

\begin{theorem} Let $X$ have type III. Then the underlying $\sts$-module structure of
the Albert algebra $\h_3(\comp)$ is given by one of the following:\\
(i)   $$\sts^9\oplus [ tr_{l_1/k}(\lb_1)\oplus  tr_{l_1/k}(\lb_1)]^3,$$
\noindent (ii) $$\sts^6\oplus [ tr_{l_1/k}(\lb_1)\oplus \lb\otimes\E\oplus \lb^\vee\otimes\E^\vee]^3,$$
for any line bundle $\lb$ over $X$,
\\ (iii) $$\sts^6\oplus [ tr_{l_1/k}(\lb_1)\oplus tr_{l/k}(\N\otimes( \E^\vee\otimes\mathcal{O}_{X_{l_2}}))]^3.$$
Here, $l/k$ is a quadratic field extension with ${\rm Gal}(l/k)=\langle\sigma\rangle$ and $\N$ a line bundle over
$X_l=X\times_kl$ not defined over $X$, which is not self-dual and satisfies $\N\cong ^\sigma\N\cong ({\rm det}\,\E
\otimes\mathcal{O}_{X_l})$.
\\ The Albert algebras in cases (i) to (iii) are mutually non-isomorphic.
 The algebras are all not isomorphic to those listed in Theorems 10 and 11.
\end{theorem}

Compared to the results in Example 1, this is a first indication that there will be many more Albert algebras over elliptic curves than over
curves of genus one.

\section{Albert algebras over elliptic curves obtained by a first Tits construction}

 For every locally free $\sts$-module $\mathcal{E}$ of constant rank 3, $\A=\E nd(\E)$ is an Azumaya algebra of rank 9.
We look at some examples, which follow directly from  [At, Theorem 8, Lemma 22], (iii) uses [AEJ3, 2.2]:

\begin{lemma} (i)  If $\E$ is absolutely indecomposable and $\E=\M\otimes\F_3$, $\M\in{\rm Pic}\,X$ a line bundle,
 then $$\A\cong \sts\oplus \F_3\oplus \F_5$$ as $\sts$-module. \\
(ii)  If $\E$ is absolutely indecomposable and $\E\in \Omega(3,d)$, ${\rm gcd}(3,d)=1$, then
$$\A\cong \sts\oplus \N_1\oplus\dots\oplus \N_m\oplus tr_{k_{1}/k}(\N_{j}')\dots\oplus tr_{k_{j}/k}(\N_j') $$
as $\sts$-module, for all line bundles $\N_1,\dots,\N_m$ of order 3 over $X$ and line bundles $\N_{i}'$ of order 3
 defined over $X_i=X\times_k k_i$ for $k_i/k$ some field extension (and not defined over $X$), $0\leq i\leq j$,
$m$ depending on $X$ and $j$ depending on $m$. \\
(iii) If $\E$ is indecomposable, but not absolutely so, then there is a  a suitable cubic field extension $l$ of $k$
and a line bundle $\N$ over $Y=X\times_k l$,  which is not defined over $X$, such that $\E=tr_{l/k}(\N)$ and $\A\cong
 tr_{l/k}(\N)\otimes tr_{l/k}(\N^{\vee})$. If $l/k$ is  Galois, then
$$\A\cong \sts^3\oplus tr_{l/k}(\N\otimes ^{\sigma_1} \N^\vee)\oplus tr_{l/k}(\N\otimes ^{\sigma_2} \N^\vee) $$
 where ${\rm Gal}(l/k)=\{id,\sigma_1,\sigma_2\}$.
 \\
(iv) If $\E=\M_1\oplus \M_2\otimes\F_2$, $\M_i\in{\rm Pic}\,X$ line bundles, then
 \begin{align*}
 \A
 & \cong
\left [\begin {array}{cc}
\sts & \E nd (\M_1, \M_2\otimes \F_2)  \\
\E nd ( \M_2\otimes \F_2,\M_1) &  \E nd ( \F_2) \\
\end {array}\right ]
\\ &
\cong \sts\oplus
\M_1\otimes \M_2^\vee\otimes\F_2\oplus \M_1^\vee\otimes \M_2\otimes\F_2\oplus \sts\oplus\F_3
 \end{align*}
 as $\sts$-module.
\\
(v) If $\E$ decomposes into the direct sum of a line bundle and an indecomposable (but not absolutely
indecomposable) vector bundle of rank 2, then there is a quadratic field extension $l/k$  with ${\rm Gal}(l/k)=\langle\sigma
\rangle$
and a line bundle $\N$ over $X_l=X\times_k l$, which is not defined over $X$, such that
 $\E=\M\oplus tr_{l/k}(\N)$ and
 \[ \A\cong
\left [\begin {array}{cc}
\sts & \h om (\M, tr_{l/k}(\N))  \\
\h om (tr_{l/k}(\N),\M) &  \E nd (tr_{l/k}(\N)) \\
\end {array}\right ]
 \]
 with
 $$\E nd (tr_{l/k}(\N))\cong  \sts^2\oplus tr_{l/k}(\N\otimes \,^{\sigma} \N^\vee).$$
 \\ (vi) If $\E$ is the direct sum of line bundles $\E=\M_1\oplus\M_2\oplus\M_3$, then
\[\A\cong
\left [\begin {array}{ccc}
\sts & \h om (\M_1, \M_2)  & \h om (\M_1,\M_3)\\
\h om (\M_2,\M_1) &  \sts & \h om (\M_2,\M_3) \\
\h om (\M_3,\M_1) & \h om (\M_3,\M_2) &  \sts \\
\end {array}\right ]
 .\]

\end{lemma}

For an Azumaya algebra $\A=\E nd(\E)$, the locally free left $\A$-modules $\Pe$ with $N_\A(\Pe)\cong\sts$
are of the kind
$$\Pe\cong \E\otimes\F^\vee$$
with $\F$ locally free of rank 3 and ${\rm det}\F\cong{\rm det}\E$. The cubic form $N$ and the adjoint $\sharp$
can be computed correspondingly [Pu2, Example 9].

\begin{example} We begin by considering the setup in Lemma 9 (i):
 Let  $\M\in{\rm Pic}\,X$ be a line bundle of degree 0 and put
$\E=\M\otimes\F_3$.
(If the degree of $\M$ is nonzero, the argument is more complicated, but the module structure is the same.)
Then  ${\rm det}\, \E\cong \M^3$.

 Let $\F$ be another locally free $\sts$-module of constant rank 3 such that
$ {\rm det}\, \F\cong \M^3$. We have the following possibilities for $\F$:
\begin{enumerate}
\item $\F$ is absolutely indecomposable, for instance, $\F=\N\otimes\F_3$ for some line bundle $\N $ of degree 0 such
that $\N^3\cong\M^3$.
\item If $\F$ is indecomposable, but not absolutely indecomposable, then there is a cubic field extension $l/k$ such that
$\F=tr_{l/k}(\lb)$ for some line bundle $\lb$ of degree 0 over $X_l$ such that $\overline{\lb}\otimes
\overline{^{\sigma_2}\lb}\otimes\overline{^{\sigma_3}\lb}\cong\overline{\M^3}$.
\item $\F= \N\oplus \G$ for some line bundle $\N$, and $ \G $ is absolutely indecomposable of rank 2.
\item  $\F= \N\oplus \G$ for some line bundle $\N$ and $ \G $ is indecomposable, but not absolutely indecomposable,
 of rank 2. Then there is a quadratic field extension $l/k$ such that $\F= \N\oplus tr_{l/k}(\lb)$
  for some line bundle $\lb$ over $X_l$. In
  particular, $\overline{\lb}\otimes \overline{^{\sigma}\lb}\cong\overline{\M^3}$.
\item If $\F$ is the direct sum of line bundles $\M_i$, then $\F= \M_1\oplus\M_2\oplus \M_1^\vee\otimes \M_2^\vee
\otimes\M^3$.
\end{enumerate}

Let $\Pe=\h om_X(\F,\E)=\E\otimes\F^\vee$:

\smallskip
\noindent
(1) If $\F=\N\otimes\F_3$ is as in (1),
then the $\sts$-module structure of the Albert algebra $\mathcal{J}(\A,\Pe,N)$ is isomorphic to

$\begin{array}{l}
 \mathcal{O}_X\oplus \F_3\oplus \F_5\oplus
  \M\otimes \N^\vee \oplus  \M\otimes \N^\vee\otimes\F_3\oplus
   \M\otimes \N^\vee \otimes\F_5 \oplus\\
 \M^\vee\otimes \N\oplus  \M^\vee\otimes \N\otimes \F_3 \oplus  \M^\vee\otimes \N\otimes \F_5.
\end{array}$

In particular, we may choose $\M$ and $\N$ both of order 3.
\\(2) If $\F$ is as in (2), then
 the $\sts$-module structure of  $\mathcal{J}(\A,\Pe,N)$ is isomorphic to
$$ \mathcal{O}_X\oplus \F_3\oplus \F_5\oplus  tr_{l/k}(\M\otimes\F_3\otimes\lb^\vee)\oplus
 tr_{l/k}(\M^\vee\otimes\F_3\otimes\lb),$$
 where $\M\otimes\F_3\otimes\lb^\vee$ and $\M^\vee\otimes\F_3\otimes\lb$ are viewed as bundles over $X_l$.
 Note that the $\sts$-modules $\Pe\cong tr_{l/k}(\M\otimes\F_3\otimes\lb^\vee)$ and $\Pe^\vee\cong
 tr_{l/k}(\M^\vee\otimes\F_3\otimes\lb)$ are indecomposable
 of rank 9.
\\(3) If $\F= \N\oplus \G$ as in (3), then
 the $\sts$-module structure of  $\mathcal{J}(\A,\Pe,N)$ is isomorphic to
$$ \sts\oplus \F_3\oplus \F_5\oplus\M\otimes\N^\vee\otimes \F_3\oplus \M\otimes \F_3\otimes\G^\vee
\oplus\M^\vee\otimes\N\otimes \F_3\oplus \M^\vee\otimes \F_3\otimes\G.$$
In particular, we may choose $\M=\sts$ in $\E$ and $\F=\N_i\oplus\N_i\otimes\F_2$ and get
$$ \sts\oplus \F_3\oplus \F_5\oplus\N_i^\vee\otimes \F_3\oplus  \F_2\oplus\F_4
\oplus\N_i\otimes \F_3\oplus  \F_2\oplus\F_4,$$
since $\F_3\otimes\F_2\cong\F_2\oplus\F_4.$
\\(4) If  $\F= \N\oplus tr_{l/k}(\lb)$ as in (4),
then the $\sts$-module structure of  $\mathcal{J}(\A,\Pe,N)$ is isomorphic to
$$ \mathcal{O}_X\oplus \F_3\oplus \F_5\oplus\M\otimes\N^\vee\otimes\F_3\oplus
tr_{l/k}(\M\otimes\F_3\otimes\lb^\vee) \oplus
\M^\vee\otimes\N\otimes\F_3\oplus tr_{l/k}(\M^\vee\otimes\F_3\otimes\lb).$$
 Note that the $\sts$-modules $ tr_{l/k}(\M\otimes\F_3\otimes\lb^\vee)$ and $
 tr_{l/k}(\M^\vee\otimes\F_3\otimes\lb)$ are indecomposable
 of rank 6.
\\
(5) If $\F$ is as in (5),
then the $\sts$-module structure of  $\mathcal{J}(\A,\Pe,N)$ is isomorphic to

$\begin{array}{l}
 \mathcal{O}_X\oplus \F_3\oplus \F_5\\
 \oplus
\M\otimes\M_1^\vee\otimes\F_3\oplus \M\otimes\M_2^\vee\otimes\F_3\oplus
\M\otimes\M_1\otimes\M_2\otimes\M^{3\vee}\otimes\F_3
\\
\oplus
\M^\vee\otimes\M_1\otimes\F_3\oplus \M^\vee\otimes\M_2\otimes\F_3\oplus
\M^\vee\otimes\M_1^\vee\otimes\M_2^\vee\otimes\M^{3}\otimes\F_3 .
\end{array}$

The Albert algebras obtained in the above cases are mutually non-isomorphic.
They are also not isomorphic to those in Theorems  10, 11, 12, 13 or 14.
\end{example}

\begin{example}
Let $l/k$ be a cubic Galois field extension with ${\rm Gal}(l/k)=\{id,\sigma_1,\sigma_2\}$  and
$\E=tr_{l/k}(\N)$ for a line bundle $\N$ which is not defined over $X$, then
$$\A\cong \sts^3\oplus tr_{l/k}(\N\otimes \,^{\sigma_1} \N)\oplus tr_{l/k}(\N\otimes \,^{\sigma_2} \N) $$
as in Lemma 9 (iii).
 Let $\F$ be another locally free $\sts$-module of constant rank 3 such that
$ {\rm det}\, \F\cong  {\rm det}\, \E $. We have the following possibilities for $\F$:

\begin{enumerate}
\item $\F\in \Omega(3,d)$ is absolutely indecomposable, $d={\rm deg}\,\E$.
(If $\N$ is a line bundle of degree 0 then $\F$
must have degree 0, hence $\F=\lb\otimes\F_3$ for some line bundle $\lb $ of degree 0 such
that $\lb^3\cong {\rm det}\, \E $, e.g., we might take $\E=tr_{l/k}(\N_i)$ for some $\N_i$ not defined over $X$, if that exists.)

\item If $\F$ is indecomposable, but not absolutely indecomposable, then there is a cubic field extension $l'/k$ such that
$\F=tr_{l'/k}(\lb)$ for some line bundle $\lb$  over $X_{l'}$, which is not defined over $X$, such that $\overline{\lb}\otimes
\overline{^{\omega_2}\lb}\otimes\overline{^{\omega_3}\lb}\cong {\rm det}\,\overline{ \E }$.
\item $\F= \M\oplus \G$ for some line bundle $\M$ and $ \G $ is absolutely indecomposable of rank 2.
\item $\F= \M\oplus \G$ for some line bundle $\M$ and $ \G $ is indecomposable, but not absolutely indecomposable,
 of rank 2. Then there is a quadratic field extension $k'/k$ such that $\F= \M\oplus tr_{k'/k}(\lb)$
  for some line bundle $\lb$ over $X'=X\times_kk'$. In
  particular, $\overline{\lb}\otimes \overline{^{\sigma}\lb}\cong {\rm det}\,\overline{ \E }$.
\item If $\F$ is the direct sum of line bundles then $\F= \M_1\oplus\M_2\oplus \M_1^\vee\otimes \M_2^\vee
\otimes {\rm det}\, \E$.
\end{enumerate}

Let $\Pe=\h om_X(\F,\E)=\E\otimes\F^\vee$.

\smallskip
\noindent
(1) Suppose that $\N$ has degree 0.
Then $\F=\lb\otimes\F_3$ for some line bundle $\lb $ of degree 0 such that $\overline{\lb}^3\cong
\overline{\N}\otimes\overline{^{\sigma_1}\N}\otimes\overline{^{\sigma_2}\N}$.
The $\sts$-module structure of $\mathcal{J}(\A,\Pe,N)$ is isomorphic to
\begin{align*}
& \sts^3\oplus tr_{l/k}(\N\otimes \,^{\sigma_1} \N)\oplus tr_{l/k}(\N\otimes \,^{\sigma_2} \N)
\\ &
\oplus tr_{l/k}(\N\otimes \lb^\vee\otimes\F_3)\oplus tr_{l/k}(\N^\vee\otimes \lb\otimes\F_3).
\end{align*}
Here, $\Pe\cong tr_{l/k}(\N\otimes \lb^\vee\otimes\F_3)$ and its dual are indecomposable of rank 9.
\\ (2) Let $\F=tr_{l'/k}(\lb)$ be as in (2).
The $\sts$-module structure of $\mathcal{J}(\A,\Pe,N)$ is given by
\begin{align*}
& \sts^3\oplus tr_{l/k}(\N\otimes \,^{\sigma_1} \N)\oplus tr_{l/k}(\N\otimes \,^{\sigma_2} \N)
\\ &
\oplus tr_{l/k}(\N)\otimes tr_{l'/k}(\lb^\vee) \oplus tr_{l/k}(\N^\vee)\otimes tr_{l'/k}(\lb).
\end{align*}

\noindent (3) Let $\F= \M\oplus \G$ be as in (3).
The $\sts$-module structure of  $\mathcal{J}(\A,\Pe,N)$ is given by
\begin{align*}
& \sts^3\oplus tr_{l/k}(\N\otimes \,^{\sigma_1} \N)\oplus tr_{l/k}(\N\otimes \,^{\sigma_2} \N)
\\ &
\oplus tr_{l/k}(\N\otimes\M^\vee) \oplus tr_{l/k}(\N\otimes\G^\vee)\oplus  tr_{l/k}(\N^\vee\otimes\M)
\oplus tr_{l/k}(\N^\vee\otimes\G).
\end{align*}
The $\sts$-module $tr_{l/k}(\N\otimes\G^\vee)$ and its dual $tr_{l/k}(\N^\vee\otimes\G)$ is indecomposable of rank 6.

\noindent (4) Let $\F= \M\oplus tr_{k'/k}(\lb)$ be as in (4).
Then the $\sts$-module structure of  $\mathcal{J}(\A,\Pe,N)$ is given by
\begin{align*}
& \sts^3\oplus tr_{l/k}(\N\otimes \,^{\sigma_1} \N)\oplus tr_{l/k}(\N\otimes \,^{\sigma_2} \N)
\\ &
\oplus tr_{l/k}(\N\otimes\M^\vee) \oplus tr_{l/k}(\N\otimes  tr_{k'/k}(\lb^\vee))\oplus  tr_{l/k}(\N^\vee\otimes\M)
\oplus tr_{l/k}(\N^\vee\otimes  tr_{k'/k}(\lb)).
\end{align*}

\noindent (5) Let $\F= \M_1\oplus\M_2\oplus \M_1^\vee\otimes \M_2^\vee
\otimes {\rm det}\, \E$.
Then the $\sts$-module structure of  $\mathcal{J}(\A,\Pe,N)$ is given by
\begin{align*}
& \sts^3\oplus tr_{l/k}(\N\otimes \,^{\sigma_1} \N)\oplus tr_{l/k}(\N\otimes \,^{\sigma_2} \N)
\\ &
\oplus tr_{l/k}(\N\otimes\M_1^\vee) \oplus tr_{l/k}(\N\otimes\M_2^\vee) \oplus tr_{l/k}(\N\otimes\M_1\otimes M_2
\otimes {\rm det}\, \E^\vee)\oplus
\\ &
\oplus tr_{l/k}(\N^\vee\otimes\M_1) \oplus tr_{l/k}(\N^\vee\otimes\M_2) \oplus tr_{l/k}(\N^\vee\otimes\M_1^\vee\otimes
M_2^\vee\otimes {\rm det}\, \E).
\end{align*}
The Albert algebras obtained in the above cases are mutually non-isomorphic. They are also non-isomorphic to the ones described
in Example 5 and not isomorphic to those in Theorems  10, 11, 12, 13 or 14.
\end{example}

\begin{example} Let $\E=\M_1\oplus \M_2\otimes\F_2$ as in Lemma 9 (iv), with $\M_i\in{\rm Pic}\,X$ line bundles. Then
 \begin{align*}
 \A
 & \cong
\left [\begin {array}{cc}
\sts & \h om (\M_1, \M_2\otimes \F_2)  \\
\h om ( \M_2\otimes \F_2,\M_1) &  \E nd ( \F_2) \\
\end {array}\right ]
 \end{align*}
with  $\E nd ( \F_2)\cong  \sts\oplus\F_3$.
 Let $\F$ be another locally free $\sts$-module of constant rank 3 such that
$ {\rm det}\, \F\cong  {\rm det}\, \E $. We have the following possibilities for $\F$:

\begin{enumerate}
\item $\F\in \Omega(3,d)$ is absolutely indecomposable, $d={\rm deg}\,\E$.
\item If $\F$ is indecomposable, but not absolutely indecomposable, then there is a cubic field extension $l'/k$ such that
$\F=tr_{l'/k}(\lb)$ for some line bundle $\lb$  over $X_{l'}$ such that $\overline{\lb}\otimes
\overline{^{\omega_2}\lb}\otimes\overline{^{\omega_3}\lb}\cong {\rm det}\,\overline{ \E }$.
\item $\F= \M\oplus \G$ for some line bundle $\M$ and $ \G $ is absolutely indecomposable of rank 2.
\item $\F= \M\oplus \G$ for some line bundle $\M$ and $ \G $ is indecomposable, but not absolutely indecomposable,
 of rank 2. Then there is a quadratic field extension $k'/k$ such that $\F= \M\oplus tr_{k'/k}(\lb)$
  for some line bundle $\M$ over $X'=X\times_kk'$. In
  particular, $\overline{\lb}\otimes \overline{^{\sigma}\lb}\cong {\rm det}\,\overline{ \E }$.
\item If $\F$ is the direct sum of line bundles then $\F= \lb_1\oplus\lb_2\oplus \lb_1^\vee\otimes \lb_2^\vee
\otimes {\rm det}\, \E$.
\end{enumerate}

Let $\Pe=\h om_X(\F,\E)=\E\otimes\F^\vee$.

\smallskip
\noindent (1)  Suppose that $\E$ has degree 0. Then $\F$ must have degree 0, hence $\F\cong\lb\otimes\F_3$ for some line bundle $\lb $ of degree 0 such
that $\lb^3\cong {\rm det}\, \E $ (e.g., take $\E=\N_i\oplus\N_i\otimes\F_2$ for some $\N_i$ of order 3 over $X$).

The $\sts$-module structure of $\mathcal{J}(\A,\Pe,N)$ is given by
\begin{align*}
 \left [\begin {array}{cc}
\sts & \h om (\M_1, \M_2\otimes \F_2)  \\
\h om ( \M_2\otimes \F_2,\M_1) &  \E nd ( \F_2) \\
\end {array}\right ]
&\\ \oplus
\M_1\otimes\lb^\vee\otimes\F_3\oplus \M_2\otimes \lb^\vee\otimes\F_2\oplus \M_2\otimes \lb^\vee\otimes\F_4
&\\ \oplus
\M_1^\vee\otimes\lb\otimes\F_3\oplus \M_2^\vee\otimes \lb\otimes\F_2\oplus \M_2^\vee\otimes \lb\otimes\F_4.
\end{align*}

\noindent (2) Let $\F=tr_{l/k}(\lb)$ be as in (2).
Then the $\sts$-module structure of $\mathcal{J}(\A,\Pe,N)$ is given by
\begin{align*}
 \left [\begin {array}{cc}
\sts & \h om (\M_1, \M_2\otimes \F_2)  \\
\h om ( \M_2\otimes \F_2,\M_1) &  \E nd ( \F_2) \\
\end {array}\right ]
&\\ \oplus
tr_{l/k}(\M_1\otimes\lb^\vee)\oplus tr_{l/k}(\M_2\otimes\F_2\otimes\lb^\vee)
&\\ \oplus
tr_{l/k}(\M_1^\vee\otimes\lb)\oplus tr_{l/k}(\M_2^\vee\otimes\F_2\otimes\lb).
\end{align*}

\noindent (3) Let $\F= \M\oplus \G$ be as in (3), say we choose $\F=\E$ for example.
Then the $\sts$-module structure of $\mathcal{J}(\A,\Pe,N)$ is given by
\begin{align*}
 \left [\begin {array}{cc}
\sts & \h om (\M_1, \M_2\otimes \F_2)  \\
\h om ( \M_2\otimes \F_2,\M_1) &  \E nd ( \F_2) \\
\end {array}\right ]
&\\ \oplus
 \left [\begin {array}{cc}
\sts & \h om_X(\M_1\otimes\F_2,\M_2)\\
 \h om_X(\M_2,\M_1\otimes\F_2) &  \E nd_X(\F_2)\\
\end {array}\right ]  & \\
\oplus
 \left [\begin {array}{cc}
\sts & \h om_X(\M_2,\M_1\otimes\F_2)\\
 \h om_X(\M_1\otimes\F_2,\M_2)  &  \E nd_X(\F_2)\\
\end {array}\right ].
\end{align*}
with $ \E nd_X(\F_2)\cong\sts\oplus\F_3$.

\noindent (4) Let $\F= \M\oplus tr_{k'/k}(\lb)$ for some line bundle $\M$ over $X'=X\times_kk'$, $k'/k$ a quadratic field extension.
Then the $\sts$-module structure of  $\mathcal{J}(\A,\Pe,N)$ is given by
\begin{align*}
 \left [\begin {array}{cc}
\sts & \h om (\M_1, \M_2\otimes \F_2)  \\
\h om ( \M_2\otimes \F_2,\M_1) &  \E nd ( \F_2) \\
\end {array}\right ]
&\\ \oplus
 \left [\begin {array}{cc}
\M_1\otimes\M^\vee & tr_{k'/k}(\M_1\otimes\lb^\vee)\\
 \h om_X(\M_2\otimes\F_2,\M) & tr_{k'/k}(\M_2\otimes\F_2\otimes\lb^\vee)\\
\end {array}\right ]  & \\
\oplus
 \left [\begin {array}{cc}
\M_1^\vee \otimes\M& tr_{k'/k}(\M_1^\vee \otimes\lb)\\
 \h om_X(\M,\M_2\otimes\F_2) & tr_{k'/k}(\M_2^\vee \otimes\F_2\otimes\lb)\\
\end {array}\right ].
\end{align*}

\noindent (5) Let $\F= \lb_1\oplus\lb_2\oplus \lb_1^\vee\otimes \lb_2^\vee \otimes {\rm det}\, \E$.
Then the $\sts$-module structure of $\mathcal{J}(\A,\Pe,N)$ is given by
\begin{align*}
 \left [\begin {array}{cc}
\sts & \h om (\M_1, \M_2\otimes \F_2)  \\
\h om ( \M_2\otimes \F_2,\M_1) &  \E nd ( \F_2) \\
\end {array}\right ]
&\\ \oplus
 \left [\begin {array}{ccc}
\M_1\otimes\lb_1^\vee & \M_1\otimes\lb_2^\vee & \M_1\otimes\lb_1\otimes\lb_2\otimes {\rm det}\,\E\\
 \h om_X(\M_2\otimes\F_2,\lb_1) &  \h om_X(\M_2\otimes\F_2,\lb_2) &  \h om_X(\M_2\otimes\F_2,\lb_1^\vee\otimes\lb_2^\vee\otimes
{\rm det}\,\E )\\
\end {array}\right ]  & \\
\oplus
 \left [\begin {array}{ccc}
\M_1^\vee\otimes\lb_1 & \M_1^\vee\otimes\lb_2 & \M_1^\vee\otimes\lb_1^\vee\otimes\lb_2^\vee\otimes {\rm det}\,\E^\vee\\
 \h om_X(\lb_1,\M_2\otimes\F_2) &  \h om_X(\lb_2,\M_2\otimes\F_2) &  \h om_X(\lb_1^\vee\otimes\lb_2^\vee\otimes
{\rm det}\,\E,\M_2\otimes\F_2 )\\
\end {array}\right ].
\end{align*}
The Albert algebras obtained in the above cases are mutually non-isomorphic. They are also non-isomorphic to the ones described
in Examples 5 and 6.
\end{example}

\begin{example}
Let $l/k$ be a quadratic field extension with ${\rm Gal}(l/k)=\{id,\sigma\}$  and $\E=\M\oplus tr_{l/k}(\N)$
as in Lemma 9 (v). Then
 \[ \A\cong
\left [\begin {array}{cc}
\sts & \h om (\M, tr_{l/k}(\N))  \\
\h om (tr_{l/k}(\N),\M) &  \E nd (tr_{l/k}(\N)) \\
\end {array}\right ]
 \]
 and
 $$\E nd (tr_{l/k}(\N))\cong  tr_{l/k}(\N\otimes \N)\oplus tr_{l/k}(\N\otimes \,^{\sigma} \N).$$
 Let $\F$ be another locally free $\sts$-module of constant rank 3 such that
$ {\rm det}\, \F\cong  {\rm det}\, \E $. We have the following possibilities for $\F$:

\begin{enumerate}
\item $\F\in \Omega(3,d)$ is absolutely indecomposable, $d={\rm deg}\,\E$.

\item If $\F$ is indecomposable, but not absolutely indecomposable, then there is a cubic field extension $l'/k$ such that
$\F=tr_{l'/k}(\lb)$ for some line bundle $\lb'$  over $X_{l'}$ such that $\overline{\lb'}\otimes
\overline{^{\omega_2}\lb}\otimes\overline{^{\omega_3}\lb}\cong {\rm det}\,\overline{ \E }$.
\item $\F= \M\oplus \G$ for some line bundle $\M$ and $ \G $ is absolutely indecomposable of rank 2.
\item $\F= \M\oplus \G$ for some line bundle $\M$ and $ \G $ is indecomposable, but not absolutely indecomposable,
 of rank 2. Then there is a quadratic field extension $k'/k$ such that $\F= \M\oplus tr_{k'/k}(\lb)$
  for some line bundle $\M$ over $X'=X\times_kk'$. In
  particular, $\overline{\lb}\otimes \overline{^{\sigma}\lb}\cong {\rm det}\,\overline{ \E }$.
\item If $\F$ is the direct sum of line bundles then $\F= \M_1\oplus\M_2\oplus \M_1^\vee\otimes \M_2^\vee
\otimes {\rm det}\, \E$.
\end{enumerate}

Let $\Pe=\h om_X(\F,\E)=\E\otimes\F^\vee$.

\smallskip
\noindent
(1) Suppose that $\E$ has degree 0. Then $\F=\lb\otimes\F_3$ for some line bundle $\lb $ of degree 0 such
that $\lb^3\cong {\rm det}\, \E $ and the $\sts$-module structure of $\mathcal{J}(\A,\Pe,N)$ is given by
\begin{align*}
\left [\begin {array}{cc}
\sts & \h om (\M, tr_{l/k}(\N))  \\
\h om (tr_{l/k}(\N),\M) &  \E nd (tr_{l/k}(\N)) \\
\end {array}\right ]
&\\ \oplus
\M\otimes\lb^\vee\otimes\F_3\oplus tr_{l/k}(\N \otimes\lb^\vee\otimes\F_3)
 & \oplus
\M^\vee\otimes\lb\otimes \F_3\oplus tr_{l/k}(\N ^\vee\otimes\lb\otimes\F_3)
\end{align*}

\noindent
(2) Let $\F=tr_{l'/k}(\lb)$.
The $\sts$-module structure of $\mathcal{J}(\A,\Pe,N)$ is given by
\begin{align*}
\left [\begin {array}{cc}
\sts & \h om (\M, tr_{l/k}(\N))  \\
\h om (tr_{l/k}(\N),\M) &  \E nd (tr_{l/k}(\N)) \\
\end {array}\right ]
&\\ \oplus
 tr_{l'/k}(\M\otimes\lb^\vee)\oplus  tr_{l/k}(\N)\otimes tr_{l'/k}(\lb^\vee)
 &\\ \oplus
 tr_{l'/k}(\M^\vee\otimes\lb)\oplus  tr_{l/k}(\N^\vee)\otimes tr_{l'/k}(\lb).
\end{align*}

\noindent
(3) Let $\F= \lb\oplus \G$ for some line bundle $\lb$ and $ \G $ is absolutely indecomposable of rank 2.
The $\sts$-module structure of $\mathcal{J}(\A,\Pe,N)$ is given by
\begin{align*}
\left [\begin {array}{cc}
\sts & \h om (\M, tr_{l/k}(\N))  \\
\h om_X (tr_{l/k}(\N),\M) &  \E nd (tr_{l/k}(\N)) \\
\end {array}\right ]
&\\ \oplus
 \left [\begin {array}{cc}
\M\otimes\lb^\vee & \M\otimes\G^\vee \\
 tr_{l/k}(\N\otimes\lb^\vee) &  tr_{l/k}(\N\otimes \G^\vee) \\
\end {array}\right ]  & \\
\oplus
 \left [\begin {array}{cc}
\M^\vee\otimes\lb & \M^\vee\otimes\G\\
 tr_{l/k}(\N^\vee\otimes\lb) &  tr_{l/k}(\N^\vee\otimes \G) \\
\end {array}\right ].
\end{align*}

\noindent
(4) Suppose that $\F=\E$. The $\sts$-module structure of $\mathcal{J}(\A,\Pe,N)$ is given by
\begin{align*}
\left [\begin {array}{cc}
\sts & \h om (\M, tr_{l/k}(\N))  \\
\h om_X (tr_{l/k}(\N),\M) &  \E nd (tr_{l/k}(\N)) \\
\end {array}\right ]
&\\ \oplus
 \left [\begin {array}{cc}
\sts & tr_{l/k}(\M\otimes\N^\vee) \\
 tr_{l/k}(\N\otimes\M^\vee) &  tr_{l/k}(\N\otimes \N^\vee) \\
\end {array}\right ]  & \\
\oplus
 \left [\begin {array}{cc}
\sts & tr_{l/k}(\M^\vee\otimes\N) \\
 tr_{l/k}(\N^\vee\otimes\M) &  tr_{l/k}(\N^\vee\otimes \N) \\
\end {array}\right ]
\end{align*}

\noindent
(5) Let $\F$ be as in (5). The $\sts$-module structure of $\mathcal{J}(\A,\Pe,N)$ is given by
\begin{align*}
\left [\begin {array}{cc}
\sts & \h om (\M, tr_{l/k}(\N))  \\
\h om_X (tr_{l/k}(\N),\M) &  \E nd (tr_{l/k}(\N)) \\
\end {array}\right ]
&\\ \oplus
 \left [\begin {array}{ccc}
\M\otimes\M_1^\vee & \M\otimes\M_2^\vee & \M\otimes\M_1\otimes\M_2\otimes {\rm det}\,\E^\vee\\
 tr_{l/k}(\N\otimes \M_1^\vee) &  tr_{l/k}(\N\otimes \M_2^\vee) &  tr_{l/k}(\N\otimes \M_1
 \otimes \M_1 \otimes{\rm det}\,\E^\vee) \\
\end {array}\right ]  & \\
\oplus
 \left [\begin {array}{ccc}
\M^\vee\otimes\M_1 & \M^\vee\otimes\M_2 & \M\otimes\M_1^\vee\otimes\M_2^\vee\otimes {\rm det}\,\E\\
 tr_{l/k}(\N\otimes \M_1^\vee) &  tr_{l/k}(\N\otimes \M_2^\vee) &  tr_{l/k}(\N\otimes \M_1
 \otimes \M_1^\vee \otimes{\rm det}\,\E) \\
\end {array}\right ]
\end{align*}
The Albert algebras obtained in the above cases are mutually non-isomorphic. They are also non-isomorphic to the ones described
in Examples 5, 6 and 7.
\end{example}

\begin{remark} We see that already by looking at the first Tits constructions starting with trivial Azumaya algebras
over an elliptic curve, we get a plenty of examples for Albert algebras with an interesting module structure.
When it comes to Azumaya algebras of a different kind, it seems harder to explicitly
compute the possible first Tits constructions.
Let, for instance, $\A=A\otimes_k \sts$ be defined over $k$.
 There is the obvious choice $\Pe\cong \N_i\otimes A\otimes_k \sts$ for a first Tits process,
Apart from that, we can use that
$A\otimes \mathcal{O}_{X'}\cong {\rm Mat}_3(\sts)$ for any cubic splitting field $k'$ of $A$,
 $X'=X\times_k k'$, to show that
a necessary condition on any left $\A$-module $\Pe$ of rank one and norm one is that
$$\Pe\otimes\mathcal{O}_{X'}\cong\F\oplus\F\oplus\F$$
for a vector bundle $\F$ of rank 3 over $X'$ with ${\rm det}\,\F\cong\mathcal{O}_{X'}$. This
 restricts the module structure of $\Pe$ somewhat,
 although many possible cases remain.

Another problem is the fact that it does not seem to be clear how
Azumaya algebras of constant rank 9 over an elliptic curve can look like, other than the obvious choices considered here.
\end{remark}

To compute examples of the Tits process seems equally difficult.
Another method to construct Albert algebras is discussed in [Pu3] and applied to find even more examples of Albert algebras over
elliptic curves. This method, which starts with a Jordan algebra of degree 4 over $X$ and looks at its trace zero elements,
gives additional information on how the first Tits process can look like in several special cases. Combined with
the first Tits construction and the Tits process it might help
to get more information on the Albert algebras over an elliptic curve. Again, one of the major difficulties is that
we do not know enough yet on Jordan algebras of degree 4 over $X$.

\smallskip
{\it Acknowledgements:}
The author would like to acknowledge the support of the ``Georg-Thieme-Ged\"{a}chtnisstiftung''
(Deutsche Forschungsgemeinschaft) during her stay at the University of Trento in 2005/6,
 and thanks the Department of Mathematics at the  University of Trento for its hospitality.

\end{document}